\documentclass[11pt,reqno,intlimits]{amsart}

\usepackage{amsfonts,amssymb,amsmath,amsthm}
\usepackage{graphicx}
\usepackage{hyperref}
\usepackage{changes}

\makeatletter

\newtheorem{theorem}{Theorem}[section]
\newtheorem{proposition}[theorem]{Proposition}
\newtheorem{lemma}[theorem]{Lemma}
\newtheorem{corollary}[theorem]{Corollary}
\theoremstyle{definition}
\newtheorem{definition}[theorem]{Definition}

\theoremstyle{remark}
\newtheorem{remark}[theorem]{Remark}
\numberwithin{equation}{section}

\@ifundefined{mathbb}{%
  \newcommand{\R}{{\rm I\mskip -3.5mu R}}
  \newcommand{\N}{{\rm I\mskip -3.5mu N}}
  
  \renewcommand{\S}{{\text{S}}}
}{%
  \newcommand{\R}{{\mathbb R}}
  \newcommand{\N}{{\mathbb N}}
  
  \renewcommand{\S}{{\mathbb S}}
}
\newcommand{\C}{{\mathcal C}}
\@ifundefined{leqslant}{}{\let\le=\leqslant  \let\leq=\leqslant}
\@ifundefined{geqslant}{}{\let\ge=\geqslant  \let\geq=\geqslant}
\newcommand{\abs}[1]{\mathopen|#1\mathclose|}
\newcommand{\biggabs}[1]{\biggl|#1\biggr|}
\newcommand{\norm}[1]{\mathopen\|#1\mathclose\|}
\newcommand{\smallvdots}{%
  {\vrule height 2.3ex depth 0.3ex width 0pt\smash{\vdots}}}
\newcommand{\e}{\varepsilon}
\let\phi=\varphi
\newcommand{\ibar}{{\bar{\text{\itshape \i}}}}
\def\l{\lambda}
\def\a{\alpha}
\def\L{\Lambda}
\def\de{\partial}
\newcommand{\dd}[2]{\frac{\mathrm{d} #1}{\mathrm{d} #2}}
\def\di12{\mathcal{D}^{1,2}(\R^n)}

\newcommand{\rad}{{\text{\upshape rad}}}
\newcommand{\intd}{\,\mathrm{d}}

\DeclareMathOperator{\Ran}{Ran}

\renewenvironment{itemize}{%
  \ifnum \@itemdepth >3 \@toodeep \else
  \advance\@itemdepth \@ne
  \edef\@itemitem {labelitem\romannumeral \the \@itemdepth}%
  \list{\csname\@itemitem\endcsname}{%
    \setlength{\itemindent}{0pt}%
    \setlength{\labelwidth}{\itemindent}%
    \setlength{\labelsep}{1ex}%
    \setlength{\itemsep}{0.1ex}%
    \setlength{\leftmargin}{2em}%
    \def\makelabel##1{\hss\llap{\upshape ##1}}%
  }\fi
}{%
  \global \advance \@listdepth \m@ne
  \endtrivlist
}

\@ifpackageloaded{hyperref}{%
  \hypersetup{%
    pdfauthor={Francesca Gladiali, Massimo Grossi, Christophe Troestler},
    pdftitle={A non-variational system involving the critical
      Sobolev exponent.  The radial case.},
    pdfpagemode=UseNone,
    pdfstartview=FitH,
    colorlinks=true,
    urlcolor=blue,
    citecolor=red,
    linkcolor=blue,
    linktocpage,
    pdfpagelabels,
    bookmarksnumbered,bookmarksopen
  }}{}

\newcommand{\FIXME}[1]{%
  \begin{hl}#1\end{hl}}

\makeatother

\title[A non-variational system involving the critical Sobolev exponent]{%
  A non-variational system involving the critical Sobolev exponent.
  The radial case.}
\thanks{The first author is supported by Gruppo Nazionale per
  l'Analisi Matematica, la Probabilit\'a e le loro Applicazioni (GNAMPA)
  of the Istituto Nazionale di Alta Matematica (INdAM). The first two
  authors are supported by PRIN-2012-grant ``Variational and
  perturbative aspects of nonlinear differential problems''.
  The third author is partially supported by the project
  ``Existence and asymptotic behavior of solutions
  to systems of semilinear elliptic partial differential equations''
  (T.1110.14)
  of the \textit{Fonds de la Recherche Fondamentale Collective},
  Belgium. }

\author[Gladiali]{Francesca  Gladiali}
\address{Francesca Gladiali, Dipartimento Polcoming, Universit\`a  di Sassari  - Via Piandanna 4, 07100 Sassari, Italy.}

\email{fgladiali@uniss.it}
\author[Grossi]{Massimo Grossi}
\address{Massimo Grossi, Dipartimento di Matematica, Universit\`a di Roma
La Sapienza, P.le A. Moro 2 - 00185 Roma, Italy.}
\email{massimo.grossi@uniroma1.it}
\author[Troestler]{Christophe Troestler}
\address{Christophe Troestler, D\'epartement de math\'ematique,
  Universit\'e de Mons, place du parc 20, B-7000 Mons, Belgium.}
\email{Christophe.Troestler@umons.ac.be}



\begin{document}

\begin{abstract}
  In this paper we consider the non-variational system
  \begin{equation}
    \begin{cases}
      \displaystyle
      -\Delta u_i = \sum\limits_{j=1}^ka_{ij}u_j^{\frac {N+2}{N-2}}
      &\text{in }\R^N,\\
      u_i>0 &\text{in }\R^N,\\
      u_i\in D^{1,2}(\R^N).
    \end{cases}
  \end{equation}
  and we give some sufficient conditions on the matrix
  $(a_{ij})_{i,j=1,\dotsc ,k}$ which ensure the existence of solution
  bifurcating from the bubble of the critical
  Sobolev equation.
\end{abstract}
\maketitle
\tableofcontents

\section{Introduction}

\subsection{Setting of the problem}

In this paper we consider the $k\times k$ system
\begin{equation} \label{y1}
  \begin{cases}
    \displaystyle
    -\Delta u_i = \sum\limits_{j=1}^ka_{ij}u_j^{2^* - 1}
    &\text{in }\R^N,\\
    u_i>0 &\text{in }\R^N,\\[1\jot]
    u_i\in D^{1,2}(\R^N),
  \end{cases}
\end{equation}
for $i=1,\dotsc ,k$,
where $N\ge 3$, $D^{1,2}(\R^N) = \bigl\{u\in L^{2^*}(\R^N)\text{ such that }
|\nabla u| \in L^2(\R^N)\bigr\}$,
$2^*=\frac{2N}{N-2}$ and the
matrix $A := (a_{ij})_{i,j=1,\dotsc ,k}$ satisfies
\begin{gather} \label{y1a}
  A \text{ is symmetric,}\\
  \label{y2}
  \sum\limits_{j=1}^k a_{ij} = 1
  \quad\text{for any } i = 1,\dotsc ,k.
\end{gather}
Assumptions \eqref{y1a} and \eqref{y2} imply that $A$ possesses
$(k-1)k/2$ free parameters, i.e., the set of matrices satisfying
\eqref{y1a}--\eqref{y2} form an affine subspace of dimension
$(k-1)k/2$.
For example, for $k=3$, such a matrix can be written
\begin{equation*}
  A = \begin{pmatrix}
    1-\a_1-\a_2&\a_1&\a_2\\
    \a_1&1-\a_1-\a_3&\a_3\\
    \a_2&\a_3&1-\a_2-\a_3
  \end{pmatrix}
  \quad \text{for } \a_1, \a_2, \a_3 \in \R.
\end{equation*}
If $u_1=u_2=\dots=u_k$ then \eqref{y1} reduces to the the
classical critical Sobolev equation
\begin{equation}
  \label{critico}
  \begin{cases}
    -\Delta u = u^{\frac{N+2}{N-2}}\quad\text{in }\R^N,\\[1\jot]
    u>0  \quad\text{in }\R^N,\\[1\jot]
    u \in D^{1,2}(\R^N).
  \end{cases}
\end{equation}
Actually, our problem can be seen as a straightforward generalization
of Equation~\eqref{critico} to the case of systems.

It is well known (see \cite{CGS}) that \eqref{critico} admits the
$(N+1)$-parameter family of solutions given by the standard bubbles
\begin{equation*}
  \label{bubble}
  U_{\delta,y}(x)
  := \frac{\left[N(N-2)\delta^2\right]^{\frac{N-2}4}}{
    (\delta^2+|x-y|^2)^{\frac{N-2}2}}
\end{equation*}
where $\delta>0$ and $y\in \R^N$. For simplicity we will denote by
\begin{equation}\label{sol}
  U(x) := U_{1,0}(x)
  = \frac {\left[N(N-2)\right]^{\frac{N-2}4}}{
    (1+|x|^2)^{\frac{N-2}2}}.
\end{equation}
In this paper we want to prove the existence of solutions to
Problem~\eqref{y1} that are different from the \textit{trivial} solution
$(U_{\delta,y},\dots, U_{\delta,y})$.\par
We now show how our system, choosing particular matrices $A$, extends
some known cases in the literature. The first example corresponds to
the matrix
$A = \bigl(\begin{smallmatrix} 0&1\\ 1&0 \end{smallmatrix}\bigr)$, so
that we have
\begin{equation}\label{y2a}
  \begin{cases}
    -\Delta u=v^{\frac {N+2}{N-2}}  &\text{in }\R^N,\\
    -\Delta v=u^{\frac {N+2}{N-2}}  &\text{in }\R^N,\\
    u,v>0 &\text{in }\R^N,\\
    u,v\in D^{1,2}(\R^N).
  \end{cases}
\end{equation}
This is a case in which the powers of the non-linearity belong to the
so-called \textit{critical hyperbola} introduced by Mitidieri
in~\cite{M1} and~\cite{M2} (see also \cite{CDM} and \cite{PV}). It is
an interesting open problem to determine whether the
system \eqref{y2a} admits nontrivial solutions.

Among the other results, we will show that the trivial solution $(U,U)$
to~\eqref{y2a} is non-degenerate, up to translations and dilations (see Theorem~\ref{lin-k}
and the ensuing discussion on page~\pageref{U-non-degenerate}).

Another interesting system which has a lot of similarities
with~\eqref{y1} is the well known \textit{Toda system}, namely
\begin{equation}
  \begin{cases}
    \displaystyle
    -\Delta u_i=\sum\limits_{j=1}^k c_{ij} \, {\operatorname e}^{u_j}
    &\text{in }\R^2,\\[5\jot]
    \displaystyle \,
    \int_{\R^2} {\operatorname e}^{u_j} <+\infty,
  \end{cases}
\end{equation}
for $i=1,\dots,k$,
where $C=(c_{ij})_{i,j=1,\dotsc,k}$ is the Cartan matrix given by
\begin{equation*}
  C = \begin{pmatrix}
    2&-1&0& \cdots &\cdots &0\\
    -1&2&-1&0&\cdots &0\\
    0&-1&2&-1& &0\\
    \vdots &\ddots&\ddots&\ddots&\ddots&\vdots&\\
    0&&\ddots &-1&2&-1\\
    0&\dots&\dots&0&-1&2
  \end{pmatrix}
\end{equation*}
Since we are considering a problem in the plane, the exponential
nonlinearity is the natural equivalent. There is a huge literature
about the Toda system, we just recall the classification result of
\cite{JW} and the paper \cite{GGW} where the matrix
$C = \bigl(\begin{smallmatrix} 2&\mu\\ \mu&2\end{smallmatrix}\bigr)$
was considered for $\mu\in(-2,0)$. Observe that if we consider the
Toda system with a general symmetric, invertible, irreducible matrix
$C$ and assume that all the entries $c_{ij}$ are positive, then (see
\cite{CK,CSW,LZ1,LZ2}) all
solutions are radial. In analogy with this result, we believe that if
all entries of $A$ are positive then all solutions to~\eqref{y1}
are radial. We do not investigate this problem in this paper (even if
we think that this problem deserves to be studied in the future) and allow some
coefficients $a_{ij}$ to be negative.  We will see that this causes
the solution set to have a richer structure.

\subsection{Main results and idea of the proof}

A basic remark is that the system~\eqref{y1} does not have a
\textit{variational structure} and so we cannot apply variational
methods. Moreover the critical powers induce a \textit{lack of
  compactness}.  Our basic tool will be the \textit{bifurcation theory}.
Solutions to~\eqref{y1} are zeros of the
functional
$F=(F_1, \dotsc, F_k) : \left(D^{1,2}(\R^N)\right)^k\to
\left(D^{1,2}(\R^N)\right)^k$ defined by
\begin{equation}
  \label{eq:F}
  F_i(u_1,\dots,u_k)
  = u_i - (-\Delta)^{-1}
  \Biggl(\sum\limits_{j=1}^ka_{ij}(u_j^+)^{\frac {N+2}{N-2}} \Biggr)
\end{equation}
for $i=1,\dots,k$.  Of course we have that $F_i(U,\dots,U)=0$ and our
aim is to find solutions close to $(U,\dots,U)$ for suitable values of
$(a_{ij})$. We will use the classical
\textit{Crandall-Rabinowitz theorem}~\cite{CRJFA}.  Its application
requires three basic ingredients:
\begin{enumerate}
\item[(i)] a good functional setting for the operator $F$,
\item[(ii)] a 1-dimensional kernel for the linearized operator $F'$,
\item[(iii)] a transversality condition.
\end{enumerate}
The lack of compactness and the rich structure of
the kernel of the linearized operator
(see Proposition~\ref{lin-k} below) make conditions (i) and
(ii) not easy to check.  (Condition (iii) will be a
straightforward computation involving the Jacobi polynomials). Now we
discuss the main points and the difficulties to be overcome in (i)
and (ii).

Let us start with the functional setting.

First of all let us note that it is not immediate to derive that our
solutions are positive. Indeed, since some of the entries $a_{ij}$ are
not necessarily positive, we cannot apply the maximum principle.
And even if they were, $F$ defined in~\eqref{eq:F}
is not smooth enough because $D^{1,2} \to D^{1,2} : u \mapsto u^+$
is not differentiable.
This problem will
be solved by restricting our operator to the subspace of
$D^{1,2}\left(\R^N\right)$ of functions decaying as $|x|^{2-N}$ at
infinity. This choice, if from one side will allow to establish the
positivity of the solution, on the other hand creates problems to
prove the compactness of the linearized operator. This will be
discussed in Section~\ref{s3}.

Now we discuss the point (ii), i.e., the linearization of $F$
around the trivial solution $(U,\dots,U)$. This leads to study the
problem,
\begin{equation} \label{linearization-2}
  \begin{cases}
    \displaystyle
    -\Delta v_i
    =\frac{N(N+2)}{\left(1+|x|^2\right)^2} \sum_{j=1}^ka_{ij}v_j
    & \text{in }\R^N,\\[1\jot]
    v_i \in D^{1,2}(\R^N),
  \end{cases}
\end{equation}
for $i=1,\dots,k$.
It will be shown that \eqref{linearization-2} can be reduced to the
classification of eigenvalues and eigenfunctions of the linearized
problem associated to the critical equation~\eqref{critico} at the
standard bubble $U$, namely,
\begin{equation} \label{eigenv}
  \begin{cases}
    \displaystyle
    -\Delta w = \lambda  \frac{N(N+2)}{(1+|x|^2)^2} w
    \quad\text{in }\R^N, \\
    w \in D^{1,2}(\R^N).
  \end{cases}
\end{equation}
It is known that $\l_0=\frac{N-2}{N+2}$ and $\l_1=1$ but nothing is
known about the other eigenvalues.  Our first result completely
describes problem~\eqref{eigenv}.  We believe that this result has
its own independent interest.
\begin{theorem}
  \label{lin-critico}
  The eigenvalues of Problem~\eqref{eigenv} are the numbers
  \begin{equation}\label{aut}
    \lambda_n=\frac{(2n+N-2)(2n+N)}{N(N+2)},\quad n\ge0.
  \end{equation}
  Each eigenvalue $\lambda_n$ has multiplicity
  \begin{equation}\label{mult}
    m(\lambda_n) = \sum\limits_{h=0}^n \frac{(N+2h-2)(N+h-3)!}{(N-2)!\,h!}
  \end{equation}
  and the corresponding eigenfunctions are, in radial coordinates,
  linear combinations of
  \begin{equation}\label{sol-lin}
    w_{n,h}(r,\theta)
    = \frac{r^h}{(1+r^2)^{h+\frac{N-2}2}}
    P_{n-h}^{\left(h+\frac{N-2}2, h+\frac{N-2}2\right)}
      \left(\frac{1-r^2}{1+r^2}\right)
    Y_h(\theta)
  \end{equation}
  for $h=0,\dots,n$, where $Y_h(\theta)$ are
  spherical harmonics related to the eigenvalue $h(h+N-2)$ and
  $P_j^{(\beta, \gamma)}$ are the Jacobi polynomials.
\end{theorem}
This result extends Theorem~11.1 in \cite{GG} where the linearized
problem of the Liouville equation in $\R^2$ at the standard bubble was
considered and highlights the role of the Jacobi polynomials as
extension of the Legendre polynomials.

Theorem~\ref{lin-critico} will be used to describe all solutions
to~\eqref{linearization-2} thanks to a change of variables to
diagonalize~$A$.
Since $A$ is symmetric, it possesses $k$
real eigenvalues $\L_1,\dotsc, \L_k$, counting algebraic multiplicity.
Assumption~\eqref{y2} implies that $1$ is always
an eigenvalue of $A$ with (at least) the eigenvectors
spanned by $(1,\dotsc,1)$.
So, without loss of generality, we can set $\L_1 = 1$.
We have the following result,
\begin{proposition}\label{lin-k}%
  Equation~\eqref{linearization-2} possesses a solution
  $v = (v_1,\dotsc,v_k) \ne (0,\dotsc,0)$ if and only if
  \begin{equation}\label{cond}
    \L_i = \lambda_n := \frac{(2n+N-2)(2n+N)}{N(N+2)}
  \end{equation}
  for some $i \in \{1,\dotsc,k\}$ and $n \in \N$.  The solutions
  coming from \eqref{cond} are
  given by
  \begin{equation}\label{eq:eigenfun}
    v = \sum_{h=0}^n c_{h} \, w_{n,h},
  \end{equation}
  where $c_{h} \in \R^k$ satisfy $A c_{h} = \lambda_n c_{h}$
  and $w_{n,h}$ are defined by~\eqref{sol-lin}.
  If several equalities of the form \eqref{cond} hold at the same
  time, the set of solutions to~\eqref{linearization-2} is the linear span
of the associated solutions of the
  form~\eqref{eq:eigenfun}.

  In particular, one always has $\L_1 = 1 = \lambda_1$ and the
  corresponding solutions are given~by
  \begin{equation}\label{1.16}
    v = c_0 \Bigl(x\cdot \nabla U+\frac{N-2}2U\Bigr)
    + \sum_{i=1}^N c_i\frac{\partial U}{\partial x_i}
  \end{equation}
  for some $c_0, c_1,\dots,c_N \in \R^k$ such that
  $A c_i = c_i$ for all $i \in \{0,1,\dotsc,N\}$.
\end{proposition}

To apply Crandall-Rabinowitz theorem, let us consider a
$\C^1$-path of matrices
\begin{equation*}
  I \subseteq \R \to \R^{k\times k} : \a \mapsto A(\a)
\end{equation*}
such that, for all $\a \in I$, $A(\a)$ satisfies
\eqref{y1a}--\eqref{y2}.  How to deal with more general situations involving two
or more parameters, can be found for example in the book~\cite{K}. Let
$\L_1(\a) = 1, \L_2(\a), \dotsc, \L_k(\a)$ be the eigenvalues of
$A(\a)$.
The previous result shows that the linearized
system~\eqref{linearization-2} has two types of degeneracies. The
first one, which holds for every value of $\a$, is due to the
invariance of the problem~\eqref{y1} under dilations and translations,
while the second one appears only at the special values
$\a$ that satisfy~\eqref{cond}.

Note that, if  $n=0$ (resp.\ $n=1$) in \eqref{cond},
i.e., if $\L_i(\a) = \frac{N-2}{N+2}$ (resp.\ $\L_i(\a) = 1$),
then $v = cU$ (resp.\
$v = c_0 \bigl(x\cdot \nabla U + \frac{N-2}2U\bigr)
+ \sum c_j\frac{\de U}{\de x_j}$) are solutions.
These are the trivial values of $\a$ and they do not provide new
solutions to problem \eqref{y1}.

Thus \eqref{cond} with $n \ge 2$ is a necessary condition
to guarantee bifurcating branches of
solutions. A similar phenomenon was previously observed in \cite{GGW}
for a general $2\times 2$ Toda system in $\R^2$.  Coming back to
Problem~\eqref{y2a}, we have $\L_1=1$ and $\L_2=-1$, so
\eqref{cond} is never satisfied if $n\neq 1$. Hence the solution $(U,U)$ is
non-degenerate (up to dilation and translation).
\label{U-non-degenerate}

Proposition \ref{lin-k} says that the kernel of the linearized
operator is composed both by radial and non-radial eigenfunctions. Of
course this is a great obstruction to applying Crandall-Rabinowitz
Theorem for which a one dimensional kernel is required. For this reason we
restrict the problem to the case of radial solutions
in $D^{1,2}(\R^N)$, i.e., we
work in the space $D^{1,2}_\rad(\R^N)$.  However we think that the
existence of non-radial eigenfunctions implies a \textit{non-radial}
bifurcation from the trivial solution.  This open problem
will be investigated in the future.

In the radial setting (see Corollary \ref{corb}), the
kernel of the linearized operator \eqref{linearization-2} has a lower
dimension.
Its dimension however depends on $A$.  Here we summarize some
sufficient conditions on the matrix $A$ such that, in the radial
setting, the kernel of the linearized operator is two-dimensional.

\vskip0.2cm \centerline{\textbf{Assumptions on the matrix $A$}}\vskip0.15cm
\bgroup
\itshape %
Let us suppose that $A$ satisfies~\eqref{y1a} and~\eqref{y2}. Moreover
assume that there exist $\bar\a$ and $\ibar \in\{2,\dots,k\}$
such that
\begin{gather}
  \L_\ibar(\bar\a) = \lambda_n
  \text{ for some } n\ge 2
  \text{ (see \eqref{cond} for the definition of } \lambda_n \text{),}
  \label{eq:ibar}
  \\
  \label{y8}
  \L_j(\bar\a) \ne \lambda_n
  \text{ for any } j\notin \{1, \ibar\} \text{ and any }n\in\N,
  \\
  \label{y9}
  \dd{\L_{\ibar}}{\a}(\bar\a) \ne 0.
\end{gather}
\egroup

Some comments on the previous assumptions: \eqref{y8} implies that
$\L_{\ibar}(\bar\a)$ is simple and then the function
$\a \mapsto \L_{\ibar}(\a)$ is smooth in a neighborhood of
$\bar\a$ (see \cite{S} for example).
Assumption~\eqref{y9} can be reformulated in terms of $\dd{A}{\a}$.
Indeed, if $e_n \ne 0$ is an unit eigenvector of $A(\bar\a)$ for the
eigenvalue $\lambda_n$, it is not difficult to show that
\begin{equation}\label{eq:dA}
  e_n \cdot \dd{A}{\a}(\bar\a) e_n = \dd{\L_i}{\a}(\bar\a).
\end{equation}
Assumption~\eqref{y8} also implies that
the kernel of the linearized operator, in this radial setting, is two-dimensional and
\eqref{y9} gives the \textit{transversality}
condition of Crandall and Rabinowitz, see~\cite{CRJFA}.
Finally, constraining
further our
operator to the orthogonal space with the radial function
$W(|x|)=\frac {1-|x|^2}{(1+|x|^2)^{N/2}}$,
we get a one dimensional kernel and so
Crandall-Rabinowitz Theorem applies. Then the point
$(\bar\a,U,\dots,U)$ is a bifurcation point when
$\bar\a$ satisfies
$\L_{\ibar}(\bar\a) = \lambda_n$ for some
$\ibar =2,\dots,k$ and some $n \ge
2$.  However this construction produces a \textit{Lagrange multiplier}
for the equations satisfied by the zeros of~$F$.

The last step is to show that this Lagrange multiplier is
$0$. In our opinion this is one of the interesting points of the
paper and it will be done using a suitable version of the Pohozaev
identity.

The Pohozaev identity was used by many authors dealing with systems of
just
two equations. The extension to the case of more equations is not
straightforward and requires the additional assumption of the invertibility of the 
matrix~$A$ (see section~\ref{Pohozaev}).
Now we are in position to state our bifurcation result.
\begin{theorem}\label{Th1}
  If $A$ satisfies \eqref{y1a} and \eqref{y2} and if
  $\bar\a$ and $\ibar$ verify the assumptions
  \eqref{eq:ibar}--\eqref{y9} and if 
   \begin{equation}  \label{y10*}
  \text{the matrix } A(\bar\a) \text{ is invertible,}
  \end{equation}
   then the point $(\bar\a, U,\dots,
  U)$ is a radial bifurcation point for the curve of trivial solutions
  $\a \mapsto (\a,U,\dots,U)$ to equation \eqref{y1}. More precisely,
  there exist a continuously differentiable curve
  defined for $\e$ small enough
  \begin{equation}\label{fin**}
    (-\e_0,\e_0) \to \R \times \bigl(D^{1,2}_\rad(\R^N)\bigr)^k :
    \e \mapsto \bigl(\a(\e), u_1(\e), \dots, u_k(\e)\bigr)
  \end{equation}
  emanating from $(\bar\a,U,\dots,U)$,
  i.e., $\bigl(\a(0),u_1(0),\dots, u_k(0)\bigr)
  = (\bar\a,U,\dots,U)$,
  such that, for every $\e \in (-\e_0,\e_0)$,
  $u(\e) = (u_1(\e), \dotsc, u_k(\e))$ is a radial solution to
  \begin{equation}
    \begin{cases}
      \displaystyle
      -\Delta u_i = \sum_{j=1}^ka_{ij}(\a) \, u_j^{\frac{N+2}{N-2}}
      &\text{in }\R^N,\\
      u_i>0 &\text{in }\R^N,\\
      u_i\in D^{1,2}(\R^N),
    \end{cases}
  \end{equation}
  with $\a = \a(\e)$.  Moreover
  \begin{equation}\label{fin*}
    u(\e)
    = (1,\dotsc, 1) \, U + \e e_n W_n\left(|x|\right)
    + \e \phi_{\e}(|x|)
  \end{equation}
  where $e_n$ is an eigenvector of $A$ for the eigenvalue
  $\L_\ibar(\bar\a) = \lambda_n$,
  $W_n$ is the function defined in \eqref{eq:def-Wn}, and
  $\phi_{\e}$ is an uniformly bounded function in
  $\bigl(D^{1,2}(\R^N)\bigr)^k$ such that $ \phi_{0}=0$.
\end{theorem}

Now let us discuss the case $k=2$. Here the number of degree of
freedom is $(k-1)k/2 =1$\linebreak[2] and the matrix
$A$ depends on a single parameter:
$A = \bigl(\begin{smallmatrix}
  \a &1-\a\\ 1-\a &\a
\end{smallmatrix}\bigr)$.  Its eigenvalues are given by
$\L_1=1$,
$\L_2=2\a-1$. It is easily seen that \eqref{eq:ibar}--\eqref{y9} are
verified with $\ibar =2$ and, in view of \eqref{cond}, with
$\a = \bar\a_{n}$ satisfying
\begin{equation}
  \bar\a_{n} = \frac{2n^2+2Nn-2n+N^2}{N(N+2)}
\end{equation}
so that the degeneracy occurs at a sequence of values $\a_n$ such that $\a_n\to +\infty$ as $n\to +\infty$.
Note that $\bar\a_{n}\ne\frac12$ for any $n\in\N$, which implies
\eqref{y10*}.  Hence Theorem
\ref{Th1} holds at the values $\a_n$ without additional assumptions and it becomes

\begin{theorem}\label{Th2}
  If $A = \bigl(\begin{smallmatrix} \a&1-\a\\
    1-\a&\a \end{smallmatrix}\bigr)$, then, for any $n\geq 2$,  the points
  $(\bar\a_{n},
  U,U)$ are radial bifurcation points for the curve of trivial solutions
  $(\a,U,U)$ to equation \eqref{y1}.  More precisely,
  there exist a continuously differentiable curve
  defined for $\e $ small enough
  \begin{equation}\label{fin1}
    (-\e_0,\e_0) \to \R \times \bigl(D^{1,2}_\rad(\R^N)\bigr)^2 :
    \e \mapsto \bigl(\a(\e), u_1(\e), u_2(\e)\bigr)
  \end{equation}
  passing through $(\bar\a_{n},U,U)$, i.e.,
  $\bigl(\a(0),u_1(0), u_2(0)\bigr) = (\bar\a_{n},U,U)$,
  such that, for all $\e \in (-\e_0,\e_0)$,
  $u_i(\e)$ is a radial solution to
  \begin{equation}
    \begin{cases}
      -\Delta u_1=\a u_1^{\frac {N+2}{N-2}} +(1-\a)u_2^{\frac {N+2}{N-2}}
      &\text{in }\R^N,\\[1\jot]
      -\Delta u_2=(1-\a)u_1^{\frac {N+2}{N-2}} +\a u_2^{\frac {N+2}{N-2}}
      &\text{in }\R^N,\\[1\jot]
      u_1,u_2>0 &\text{in }\R^N,\\
      u_1,u_2\in D^{1,2}(\R^N),
    \end{cases}
  \end{equation}
  with $\a = \a(\e)$. Moreover,
  \begin{equation}
    \begin{split}
      & u_1(\e)
      = U + \e W_n\left(|x|\right)
      + \e \phi_{1,\e}(|x|)\\
      &u_2(\e)
      = U - \e W_n\left(|x|\right)
      + \e \phi_{2,\e}(|x|)
      \end{split}
  \end{equation}
  where $W_n$ is the function defined in \eqref{eq:def-Wn}, and
  $\phi_{1,\e},\phi_{2,\e}$ are functions uniformly bounded in
  $D^{1,2}(\R^N)$ and such that $ \phi_{s,0}=0$ for $s=1,2$.
\end{theorem}
The same type of result was proved in \cite{GGW} for the Toda system
in~$\R^2$.

\subsection{Extensions and related problems}

A first interesting question that arises from Proposition \ref{lin-k}
concerns the existence of non-radial solutions.

\vskip 0.1cm plus 0.1cm \textit{%
  Question 1. Do exist nonradial solutions to \eqref{y1} bifurcating
  from the values $\a$ which verify \eqref{cond}? How
  many?}\vskip0.1cm

In analogy with the classification result of Jost-Wang (\cite{JW}),
it is possible to think that the number of solutions of \eqref{y1}
coincides with that of the linearized operator (see also
\cite{WZZ}). In our case, if for example $k=2$, we would have the
existence of at least
$\sum\limits_{h=0}^n \frac{(N+2h-2)(N+h-3)!}{(N-2)!\,h!}$ distinct
solutions.

\pagebreak[2]
\vskip 0.1cm plus 0.1cm

Another interesting question concerns the shape of the branch of our
solutions.

\vskip 0.1cm plus 0.1cm \textit{%
  Question 2.  What about the bifurcation diagram for $\a(\e)$
  close to $\a$?}
\vskip0.1cm

This question is quite delicate and the answer seems to strongly
depend on $A$. In Appendix~\ref{A}, we carry out the computation of
the first derivative of $\a(\e)$ with respect to $\e$. It is worth
noting that if $k=2$ then $\dd{\a}{\e}(0) = 0$ (then nothing can be
said about the behaviour of the branch) but, if $k>2$, we can have that
$\dd{\a}{\e}(0) \ne 0$. In this case, the bifurcation is
\textit{transcritical}.

\pagebreak[3]
\vskip 0.1cm \textit{%
  Question 3.  What possible extensions may be considered?}
\vskip0.1cm
Another interesting problem to which one can apply our techniques is
given by the Gross-Pitaevskii type systems, namely,
\begin{equation}
  \begin{cases}
    \displaystyle
    -\Delta u_i
    = \Biggl(\sum_{j=1}^ka_{ij}u_j^2\Biggr)^{\frac2{N-2}} u_i
    & \text{in }\R^N,\\
    u_i>0& \text{in }\R^N,\\
    u_i\in D^{1,2}(\R^N),
  \end{cases}
\end{equation}
When $A = \bigl(\begin{smallmatrix} 1&1\\
  1&1 \end{smallmatrix}\bigr)$,
this problem was studied in \cite{DH} as the limit problem for blowing up
solution on Riemannian manifolds. The authors proved that only the trivial
solution $(U,\dots,U)$ exist.

The paper is organized as follows.  In Section~\ref{sec:linearization},
we prove Theorem~\ref{lin-critico}
and we study the linearization of our system at the trivial
solution. In Section~\ref{sec:bifurcation} we define the functional
setting, we apply the
Crandall-Rabinowitz result, and we prove the Pohozaev identity getting
our bifurcation result, Theorem~\ref{Th1}. Finally in the appendix, we give
some examples of possible behaviour of the branches.

\section{Linearization at the standard bubble}
\label{sec:linearization}

\subsection{The case of a single equation}\label{s2}

Let us consider the critical equation \eqref{critico} and the
associated eigenvalue problem \eqref{eigenv}.
It is well known that the first eigenvalue to \eqref{eigenv} is given
by $\lambda_0 = \frac{N-2}{N+2}<1$ and the corresponding eigenfunction
is $U$, while the second eigenvalue is $\lambda_1 = 1$ and it has an
$N+1$ dimensional kernel spanned by
$\frac{\partial U}{\partial x_1}, \dots,\frac{\partial U}{\partial
  x_N}, x\cdot \nabla U + \frac{N-2}2U$
(see for example \cite{BE} or \cite{AGAP}).  In this section,
we compute all
eigenvalues and corresponding eigenfunctions of~\eqref{eigenv}.

\begin{proof}[Proof of Theorem \ref{lin-critico}]
  We decompose the  solutions to~\eqref{eigenv} using spherical
  harmonics:
  \begin{equation*}
    w(r,\theta) = \sum_{h=0}^{\infty} \psi_h(r)Y_h(\theta),
    \qquad \text{where }
    r=|x| , \  \theta=\frac{x}{|x|} \in \S^{N-1},
  \end{equation*}
  and
  \begin{equation*}
    \psi_h(r) = \int_{\S^{N-1}}w(r,\theta) Y_h(\theta)\intd\theta.
  \end{equation*}
  Here $Y_h(\theta)$ denotes a $h$-th spherical
  harmonic which satisfies:
  \begin{equation*}
    -\Delta_{\S^{N-1}} Y_h = \beta_hY_h
  \end{equation*}
  where $\Delta_{\S^{N-1}}$ is the Laplace-Beltrami operator on
  $\S^{N-1}$ with the standard metric and $\beta_h$ is the $h$-th
  eigenvalue of $-\Delta_{\S^{N-1}}$. It is known that
  \begin{equation*}
    \beta_h = h(N-2+h), \qquad h=0,1,2,\dots
  \end{equation*}
  and its multiplicity is
  \begin{equation*}
    m(\beta_h)= \frac{(N+2h-2)(N+h-3)!}{(N-2)!\,h!} .
  \end{equation*}
  By standard regularity theory, the function $w\in D^{1,2}(\R^N) $ is
  a solution of \eqref{eigenv} if and only if $\psi_h(r)$ is a weak
  solution of
  \begin{align}
    \label{1.8}
    \begin{cases}
      -\psi_h''(r) - \frac{N-1}{r}\psi_h'(r)
      + \frac{\beta_h}{r^2}\psi_h(r)
      = \lambda \frac{N(N+2)}{(1+r^{2})^2}\,\psi_h(r)\,,
      \quad\text{in } (0,\infty) \\[1\jot]
      \int_0^{+\infty}r^{N-1}|\psi_h'(r)|^2 \intd r< +\infty \\[1\jot]
    \end{cases}
  \end{align}
  We shall solve \eqref{1.8} using the following transformation
  \begin{equation*}
    \psi_h(r)=r^{-\frac {N-2}2}B_h(r).
  \end{equation*}
  The function $B_h$ solves
  \begin{equation}
    \label{a1*}
    -B_h''(r)-\frac 1r B_h'(r)
    =- \nu_h^2 \frac 1{r^2}B_h(r)
    + \frac{\lambda N(N+2)}{(1+r^2)^2}B_h(r)
    \quad\text{in } (0,\infty)
  \end{equation}
  where $\nu_h := h + \frac{N-2}{2}$,
  and
  \begin{equation}\label{1.8b}
    \int_0^{\infty}r\left(B_h'(r)\right)^2+r^{-1}
    \bigl(B_h(r)\bigr)^2 \intd r <\infty.
  \end{equation}
  The proof of \eqref{a1*} is a straightforward computation, so let us
  prove \eqref{1.8b}. By the definition of $\psi_h$ we get
  \begin{equation}\nonumber
    r^{N-1}|\psi_h'|^2
    = r(B'_h)^2+\left(\frac{N-2}2\right)^2\frac{B_h^2}r
    - (N-2) B_h'B_h.
  \end{equation}
  Now, integrating between $\e_n$ and $R_n$ (the sequences $\e_n$
  and $R_n$ will be chosen later) we get
  \begin{equation*}
    \int_{\e_n}^{R_n}r^{N-1}|\psi_h'|^2 \intd r
    = \int_{\e_n}^{R_n} r(B'_h)^2
    + \left(\frac{N-2}2\right)^2 \frac{B_h^2}r  \intd r
    - \frac{N-2}2 \bigl(B_h^2(R_n) - B_h^2(\e_n)\bigr).
  \end{equation*}
  Since $\psi_h \in L^{2^*}\bigl([0,\infty); r^{N-1}\intd r\bigr)$,
  there exist sequences $\e_n \to 0$ and $R_n \to +\infty$ such that
  \begin{equation}\label{1-bis}
    R_n^{\frac{N-2}2} \abs{\psi_h(R_n)} \to 0\  \text{and }
    \e_n^{\frac {N-2}2} \abs{\psi_h(\e_n)} \to 0
    \quad\text{ as }n\to +\infty.
  \end{equation}
  Finally, by \eqref{1-bis}, we deduce that
  $B_h(R_n)$ and $B_h(\e_n)$ go to zero and this, together with
  \eqref{1.8}, gives \eqref{1.8b}.

  Then, from Lemma~2.4 of~\cite{GGN}, we have that $B_h(0)=0$ and
  \begin{equation}
    \label{asimpt-zero}
    B_h(r) = O\bigl(r^{\nu_h}\bigr)
    \quad\text{as } r \to 0.
  \end{equation}
  Let us show that an analogous estimate holds at infinity. To do this set
  $C(r) := B_h\bigl(\frac{1}{r}\bigr)$. So the claim follows if we prove
  that $C$ is bounded near the origin. A straightforward computation
  proves that $C$ satisfies again \eqref{a1*} and \eqref{1.8b}.
  Then by (2.28) of Lemma 2.4 of \cite{GGN} we get that
  \begin{equation}\nonumber \label{a3*}
    C(r) = O\bigl(r^{\nu_h}\bigr)
    \quad\text{as } r \to 0,
  \end{equation}
  which gives
  \begin{equation}\label{asimpt-infty}
    B_h(r) = O\bigl(r^{-\nu_h}\bigr)
    \quad\text{as } r \to +\infty
  \end{equation}
  and so the claim follows.

  Setting $\xi := \frac{1-r^2}{1+r^2}$, letting $R_h(\xi) := B_h(r)$, and
  using the definition of $\beta_h$ we have
  \begin{equation}
    \label{Legendre}
    \frac{\partial}{\partial \xi}
    \left((1-\xi^2)\frac{\partial R_h}{\partial\xi}\right)
    + \left( \lambda \frac{N(N+2)}{4} - \frac{\nu_h^2}{1-\xi^2}
    \right)R_h(\xi)
    = 0
  \end{equation}
  for $-1< \xi< 1$.
  Now, let us set
  $A_h(\xi) := (1-\xi^2)^{-\nu_h/2} R_h(\xi)$. Then $A_h(\xi)$
  solves, for $\xi\in (-1,1)$,
  \begin{equation}\label{a3}
    (1-\xi^2)A_h''(\xi) - 2 (1+\nu_h) \xi A_h'(\xi)
    + \left[\frac{N(N+2)}4\lambda -\nu_h^2-\nu_h \right]A_h(\xi) = 0.
  \end{equation}
  This is a particular case of the Jacobi equation
  \begin{equation}\nonumber
    (1-\xi^2)y'' + \bigl[\gamma-\beta -(2+\beta+\gamma)\xi\bigr] y'
    + m(1+\beta+\gamma+m)y = 0
  \end{equation}
  with $\beta = \gamma = \nu_h$.
  From \eqref{asimpt-zero} we have that
  $r^{-\nu_h}B_h(r) = R_h(\xi)
  \Bigl(\frac{1+\xi}{1-\xi}\Bigr)^{{\nu_h}/2}$
  is bounded near $\xi=1$, while from \eqref{asimpt-infty} we have
  that
  $r^{\nu_h}B_h(r) = R_h(\xi) \Bigl(\frac{1-\xi}{1+\xi}\Bigr)^{{\nu_h}/2}$
  is bounded near $\xi=-1$. This implies that the solution $A_h(\xi)$
  of \eqref{a3} is bounded at $\xi=\pm 1$.

  It is known that \eqref{a3} admits a bounded solution if and only if
  \begin{equation}\nonumber
    \frac{N(N+2)}4 \lambda - \nu_h^2 - \nu_h = m(m + 2\nu_h + 1)
  \end{equation}
  for $m=0,1,\dots$\@  Inserting the definition of $\nu_h$ we get that
  \begin{equation}\nonumber
    \lambda =\frac4{N(N+2)}\left((m+h)^2+(m+h)(N-1)+\frac{N(N-2)}4\right)
  \end{equation}
  and, setting $n := m+h\in\N$, we derive
  \begin{equation}\nonumber
    \lambda =\frac{(2n+N-2)(2n+N)}{N(N+2)}.
  \end{equation}
  Moreover the bounded solutions
  of \eqref{a3} corresponding to $m(m+2\nu_h+1)$ are given by
 multiples of
  $A_{m,h}(\xi)=P_m^{(\nu_h, \nu_h)}(\xi)$ where $P_m^{(\gamma, \beta)}$ are
  the Jacobi polynomials:
  \begin{equation*}
    P_m^{(\beta, \gamma)}(\xi)
    = \sum_{s=0}^m \binom{m+\beta}{s} \binom{m+\gamma}{m-s}
    \left(\frac{\xi-1}2\right)^{m-s} \left(\frac{\xi+1}2\right)^s .
  \end{equation*}
  The polynomials $P_m^{(\nu_h, \nu_h)}$ are also known as the
  \textit{Gegenbauer polynomials} or the \textit{ultraspherical
    polynomials}.  They form a basis of the space
  $L^2\bigl((-1,1); (1-\xi^2)^{\nu_h}\intd \xi\bigr)$ (see p.~202 in
  \cite{LTWZ} for example).  Moreover we have that
  \begin{align*}
    R_{m,h}(\xi)
    &= (1-\xi^2)^{\nu_h/2} \, P_m^{(\nu_h, \nu_h)} (\xi), \\
    B_{m,h}(r)
    &= \frac{(2r)^{\nu_h}}{(1+r^2)^{\nu_h}}
      P_m^{(\nu_h, \nu_h)} \biggl(\frac{1-r^2}{1+r^2}\biggr)
  \end{align*}
  and
  \begin{equation*}
    \psi_{m,h}(r)
    = \frac{2^{h+\frac{N-2}{2}} \, r^h}{(1+r^2)^{h+\frac{N-2}{2}}}
    P_m^{(h+\frac{N-2}2, h+\frac{N-2}2)} \biggl(\frac{1-r^2}{1+r^2}\biggr).
  \end{equation*}
  Finally, recalling that $n=m+h$ then \eqref{sol-lin} follows. The
  multiplicity \eqref{mult} then follows counting the multiplicity of
  the spherical harmonics.
\end{proof}

\begin{remark}
  Note that Theorem \ref{lin-critico} also holds when $N=2$. In this
  case we have that $\l_n=\frac {n(n+1)}2$. These are exactly the
  eigenvalues associated to the linearization of the classical
  Liouville problem
  \begin{equation*}
    \begin{cases}
      \displaystyle
      -\Delta U = {\operatorname e} ^U& \text{in }\R^2\\
      \int_{\R^2} {\operatorname e}^U \intd x<\infty
    \end{cases}
  \end{equation*}
  at the standard bubble
  $U_{\R^2}(x)=\log\frac{64}{(8+|x|^2)^2}$. This result was proved
  in~\cite{GG} (see Theorem~11.1) and the corresponding eigenfunctions
  are spanned by
  $(P_L)^h_{n-h}\left(\frac{8-r^2}{8+r^2}\right) Y_h(\theta)$ where
  $(P_L)$ are the Legendre polynomials and $Y_h(\theta)$ are the
  spherical harmonics in $\R^2$. Observe that for $N=2$ the Jacoby
  polynomials
  $P_{n-h}^{\left(h+\frac{N-2}2,
      h+\frac{N-2}2\right)}\left(\frac{1-r^2}{1+r^2}\right)$
  in \eqref{sol-lin} become the Legendre polynomials and so Theorem
  \ref{lin-critico} contains also the result of \cite{GG} for~$\R^2$.
\end{remark}

\subsection{The case of the system}\label{s3}

Now we are in position to prove Proposition \ref{lin-k} in the
Introduction.
\begin{proof}[Proof of Proposition \ref{lin-k}]
  Because $A$ is symmetric, there exists an orthogonal matrix
  $B=(b_{ij})$ such that
  \begin{equation}\label{y4}
    B^{-1}AB = \L
  \end{equation}
  and $\L$ is the diagonal matrix with the eigenvalues
  $(\L_1,\dotsc,\L_k)$ as diagonal elements.
  Let $w_i := \sum_{j=1}^kb_{ij}^{-1}v_j = \sum_{j=1}^kb_{ji}v_j$.
  Then $w=(w_1,\dots,w_k)$ is a solution to
  \begin{equation} \label{y5}
    \forall i=1,\dots,k,\quad
    \begin{cases}
      -\Delta w_i =\frac{N(N+2)}{\left(1+|x|^2\right)^2}\L_i w_i
      & \text{in }\R^N,\\[1\jot]
      w_i \in D^{1,2}(\R^N).
    \end{cases}
  \end{equation}
  By our choice of the matrix $B$ in \eqref{y4} the $k$ equations in
  \eqref{y5} are decoupled and so we can solve them independently.
  Remember that we have set $\L_1 = 1$.  So the first equation
  of~\eqref{y5} reads
  \begin{equation}\label{w_1}
    {-\Delta} w_1=\frac{N(N+2)}{(1+|x|^2)^2} w_1 \quad \hbox{ in }\R^N
  \end{equation}
  and it is well known it
  admits a nontrivial solution which is a linear combination of
  $\frac{\partial U}{\partial x_i}$ for $i=1,\dots,N$ and
  $x\cdot \nabla U+\frac{N-2}2 U$.
  It is important to observe that equation \eqref{w_1} does not depend
  on $\a$ and so system \eqref{y5} has the solution in \eqref{1.16}
  for every value of~$\a$.

  The other equations in \eqref{y5} have a
  nontrivial solution if and only if $\L_i$ is an eigenvalue of
  problem \eqref{eigenv}, i.e., using Theorem \ref{lin-critico} if and
  only if \eqref{cond} is satisfied for some $i=2,\dots,k$, for some
  $n\in \N$ and for some value of $\a\in \R$. When \eqref{cond} is
  satisfied the the $i$-th equation in system \eqref{y5} has as a
  solutions a linear combination of the eigenfuntions of
  \eqref{eigenv} related to the eigenvalue $\l_n$, and so \eqref{eq:eigenfun}
  follows.
\end{proof}
One of the main hypothesis of the bifurcation result is that the
kernel of the linearized operator has to be one dimensional.  From the
previous result, we know that the linearized operator has instead a
very rich kernel. To overcome this problem we consider only the case
of radial solutions in $D^{1,2}(\R^N)$, that is, we will work in the space
$D^{1,2}_\rad(\R^N)$.  First we state the result of Proposition
\ref{lin-k} in this radial setting.
\begin{corollary}\label{corb}
  Equation~\eqref{linearization-2} possesses a \textit{radial} solution
  $v = (v_1,\dotsc,v_k) \ne (0,\dotsc,0)$ if and only if
  $\L_i = \lambda_n$
  for some $i \in \{1,\dots,k\}$ and $n\in\N$, in which case the
  associated radial solutions are given by
  $v= c W_n$ where $c \in \R^k$ is an eigenvector associated to the eigenvalue
  $\L_i$ and
  \begin{equation}\label{eq:def-Wn}
    W_n(|x|)
    := \frac{1}{(1+|x|^2)^{\frac{N-2}2}} \,
    P_{n}^{\left(\frac{N-2}{2}, \frac{N-2}{2}\right)}
    \biggl(\frac{1-|x|^2}{1+|x|^2}\biggr).
  \end{equation}
  Here
  \begin{math}
    P_{n}^{\left(\frac{N-2}{2}, \frac{N-2}{2}\right)}(\xi)
    = \sum\limits_{s=0}^n
    \binom{n+\frac{N-2}2}{s} \binom{n+\frac{N-2}2}{n-s}
    \bigl(\frac{\xi-1}2\bigr)^{n-s} \bigl(\frac{\xi+1}2\bigr)^s
  \end{math}. %
  If several equalities of the form $\L_i = \lambda_n$ hold at the same
  time, the set of radial solutions to~\eqref{linearization-2}
  is the linear span of the associated solutions of
  the form~\eqref{eq:def-Wn}.

  In particular, if $\L_1 = 1$ is a simple eigenvalue of $A$
  (which is the case under
  assumption~\eqref{y8}), all corresponding radial solutions
  to equation~\eqref{linearization-2} are given by multiples~of
  \begin{equation}\label{eq:def-W}
    (1,\dotsc, 1) \, W
    \quad\text{where }
    W(|x|) := x\cdot \nabla U + \tfrac{N-2}{2} U
    = d \, \frac {1-|x|^2}{(1+|x|^2)^{N/2}}
  \end{equation}
  and $d = \tfrac{1}{2} N^{(N-2)/4} (N-2)^{(N+2)/4}$.
\end{corollary}
\begin{proof}
  Restricting to the radial setting we have that equation \eqref{w_1}
  admits only the solution in \eqref{eq:def-W}.  From Theorem
  \ref{lin-critico} instead follows that any eigenvalue $\l_n$ of the
  critical problem admits one radial solution wich is the one in
  \eqref{sol-lin} that correspond to $h=0$. Then \eqref{eq:def-Wn} follows
  from \eqref{eq:eigenfun}.
\end{proof}

\section{The bifurcation result}
\label{sec:bifurcation}

\subsection{The functional setting}

As mentioned in the introduction, the proof of our bifurcation result requires an appropriate functional setting which is a delicate part of
the proof.

Both for the lack of differentiability of $u \mapsto u^+$ and the
difficulty of proving that the solution $(u_1,\dots,u_k)$ is positive,
we need to restrict to a subset of $D^{1,2}(\R^N)$ with
a stronger topology.
As before let
$D^{1,2}_\rad(\R^N) = \bigl\{u\in D^{1,2}(\R^N) \bigm|
u=u(|x|)\bigr\}$ and set
\begin{equation*}
  D := \Bigr\{u\in L^{\infty}(\R^N) \Bigm|
  \sup\limits_{x\in\R^N}\frac{|u(x)|}{U(x)}<+\infty \Bigr\}
\end{equation*}
endowed with the norm $\norm{u}_D := \sup_{x\in\R^N}\frac{|u(x)|}{U(x)}$
and we define
\begin{equation*}
  X = D^{1,2}_\rad(\R^N)\cap D.
\end{equation*}
Then $X$ is a Banach space equipped with the norm
$\norm{u}_X := \max\{\norm{u}_{1,2} ,\norm{u}_D\}$ where
$\norm{u}_{1,2}$ is the classical norm on $D^{1,2}(\R^N)$.

To rule out the degeneracy due to the invariance under dilations of
Problem~\eqref{y1}, we will solve the linearized equation in
the subspace of functions that are
orthogonal in $\bigl(D^{1,2}_\rad(\R^N)\bigr)^k$ to
$(1,...,1) W(|x|)$ defined in \eqref{eq:def-W}.  Let $P_ {K}$ be the
orthogonal projection (with respect to the inner product of
$\bigl(D^{1,2}(\R^N)\bigr)^k$)
from $X^k$ onto the subspace $K$
given by
\begin{equation}\label{K}
  K := \biggl\{ g\in X^k \biggm|
  \sum_{i=1}^k \, \int_{\R^N}
  \nabla W \cdot \nabla g_i(x) \intd x = 0
  \biggr\}.
\end{equation}
\begin{definition}
  Let us denote by
  $B := \bigl\{ u\in X \bigm| \norm{u - U}_X < \frac1{2} \bigr\}$ and define
  the operator
  \begin{equation*}
    T : \R \times (K\cap B^{k}) \to K
  \end{equation*}
  as
  \begin{equation}
    \label{T}
    T(\a, u_1,\dots, u_k) :=
    P_K
    \begin{pmatrix}
      \displaystyle
      u_1 - (-\Delta)^{-1} \sum_{j=1}^k a_{1j}(\a)\, u_j^{2^*-1}\\
      \vdots\\
      \displaystyle
      u_k - (-\Delta)^{-1} \sum_{j=1}^k a_{kj}(\a)\, u_j^{2^*-1}\\
    \end{pmatrix}
  \end{equation}
  Note that, since $u_i \in B$,
  $u_i = U + (u_i - U) > \frac{1}{2} U$
  is positive so that $u_i^{2^*-1}$
  is well defined for any $N \ge 3$.
\end{definition}

The zeros of the operator $T$ satisfy
\begin{equation}
  \label{2-rev}
  \begin{cases}
    \displaystyle
    -\Delta u_i = \sum_{j=1}^k a_{ij}(\a) \, u_j^{2^*-1}
    + L \, \frac{N(N+2)}{(1+|x|^2)^{2}} \, W
    & \text{in }\R^N \text{, for } i=1,\dotsc,k,\\
    u = (u_1, u_2, \dotsc, u_k) \in K\cap B^{k},
  \end{cases}
\end{equation}
where $L=L(u)\in \R$ is a Lagrange multiplier. Once we prove the
existence of $(u_1,\dots,u_k)$ that satisfies \eqref{2-rev}, the final
step will be to show that $L=0$ so that $u$ is indeed a solution
to~\eqref{y1} and this will be done in the next section using a
\textit{Pohozaev} identity.

First we prove some properties of the operator $T$.

\begin{lemma}\label{cont-PK}
  The projector $P_K : X \to X$ is well defined and continuous.
\end{lemma}
\begin{proof}
  Let $W^* := W / \norm{W}_{1,2}$.  One can write
  $P_K u = u - (1,\dotsc,1) W^* \sum_{i=1}^k (u_i|W^*)_{D^{1,2}} $.
  From \eqref{eq:def-W}, one easily
  shows that there exists a $C \in \R$ such that $\abs{W^*} \le C U$
  and so $\norm{W^*}_X < +\infty$.  The statement readily follows from
  these facts.
\end{proof}

\begin{lemma}\label{lemma1}
  The operator $T$ in \eqref{T} is continuous from
  $\R \times (K\cap B^{k})$ into $K $ and its derivatives
  $\partial_{\a} T$, $\partial_u T$ and $\partial_{\a u}T$ exist
  and are continuous.
\end{lemma}
\begin{proof}
  Since $u_j$ belongs to $X$, we have that
  $u_j^{2^*-1} \in L^{\frac{2N}{N+2}}_\rad(\R^N)$ for
  any $j=1,\dots,k$. As a consequence, there exists a unique
  $g_i\in D^{1,2}_\rad(\R^N)$ for $i=1,\dots,k$ such that $g_i$ is a
  weak solution to $-\Delta g_i=f_i$ in $\R^N$ where
  \begin{equation}\label{eq:fi}
    f_i := \sum_{j=1}^k a_{ij}(\a)\, u_j^{2^*-1}
  \end{equation}
  The solution $g_i$ enjoys the following representation:
  \begin{equation*}
    g_i(x) = \frac 1{\omega_N(N-2)}
    \int_{\R^N}\frac 1{|x-y|^{N-2 }}f_i(y)\intd y
    \text{,}
  \end{equation*}
  where $\omega_N$ is the area of the unit sphere in $\R^N$.  By
  assumption $u_{i}\in B$ and this implies that
  $|f_i(x)|\leq CU^{2^*-1} (x)$ so that
  \begin{equation*}
    \abs{g_{i}(x)}
    \le C \int_{\R^N}\frac 1{|x-y|^{N-2}} U^{2^*-1}(y) \intd y
    = C U(x)
  \end{equation*}
  and $g_i\in X$.  (Different occurrences of $C$ may denote
  different constants.)
  To prove the continuity of $T$ in $K\cap B^{k}$, let
  $\a_n \to \a$ in $\R$ and
  $u_{i,n} \to u_i$ in
  $X$ (for $i=1,\dots,k$) as $n \to \infty$, and set
  \begin{equation*}
    g_{i,n}
    := (-\Delta)^{-1} f_{i,n}
    \quad\text{where }
    f_{n,i}
    := \sum_{j=1}^k a_{ij}(\a_n)\, u_{j,n}^{2^*-1} .
  \end{equation*}
  Since $u_{i,n} \to u_i$ in $D^{1,2}(\R^N)$, the convergence also
  holds in $L^{2^*}(\R^N)$.  Using Lebesgue's dominated convergence
  theorem and its converse, one deduces that $f_{i,n} \to f_i$ in
  $L^{\frac{2N}{N+2}}$ where $f_i$ is defined as in~\eqref{eq:fi}.
  Therefore $g_{i,n} \to g_i$ in $D^{1,2}$ and
  $T(\a_n, u_n) \to T(\a, u)$ in~$D^{1,2}$.

  Now let us show
  the convergence in $D$.
  We have that
  \begin{align*}
    \frac{|g_{i,n}(x)-g_i(x)|}{U(x)}
    &\le \frac{1}{\omega_N(N-2)U(x)}
      \int_{\R^N}\frac{1}{|x-y|^{N-2 }}
      \frac{\abs{f_{i,n}(y) - f_i(y)}}{U(y)^{2^*-1}}
      U(y)^{2^*-1} \intd y\\
    &\le C\sup_{y \in \R^N}
      \frac{\abs{f_{i,n}(y) - f_i(y)}}{U(y)^{2^*-1}}.
  \end{align*}
  Moreover,
  \begin{multline*}
    \sup_{y \in \R^N}
    \frac{\abs{f_{i,n}(y) - f_i(y)}}{U(y)^{2^*-1}}
    \le \sum_{j=1}^k \abs{a_{ij}(\a_n) - a_{ij}(\a)}
    \sup_{y \in \R^N} \biggl( \frac{\abs{u_j(y)}}{U(y)} \biggr)^{2^*-1}\\
    + \sum_{j=1}^k \abs{a_{ij}(\a_n)} \sup_{y \in \R^N}
    \biggabs{ \biggl(\frac{u_{j,n}(y)}{U(y)} \biggr)^{2^*-1}
      - \biggl( \frac{u_j(y)}{U(y)} \biggr)^{2^*-1} } .
  \end{multline*}
  The first term goes to $0$ because $a_{ij}(\a_n) \to a_{ij}(\a)$ and
  $\abs{u_j} \le C U$.  As
  $\bigl(a_{ij}(\a_n)\bigr)_n$ are bounded sequences,
  it is enough to show that the
  last factor goes to~$0$.  This is the case because, thanks to
  the convergence in $D$,
  ${u_{j,n}}/{U} \to {u_{j}}/{U}$ uniformly for all $j$ and the map
  $\zeta \mapsto \zeta^{2^*-1}$ is continuous.

  The existence of $\partial_{\a}T$, $\partial_{u}T$
  and $\partial_{\a u}T$ (for the
  topology of $X$) and their continuity follows in a very similar way
  and we omit~it.
\end{proof}

Next we show a compactness result for the operator
$w \mapsto (-\Delta)^{-1}\Bigl( \frac w{(1+|x|^2)^2} \Bigr)$.  
We need some
decay estimates on solutions of a semilinear elliptic equation.

\begin{lemma}\label{lemma-ST}
  If $0 < p < N$ and $h$ is a nonnegative, radial function
  belonging to $L^1(\R^N)$, then
  \begin{equation*}
    \int_{\R^N} \frac{h(y)}{|x-y|^{p}} \intd y
    = O\biggl( \frac{1}{\abs{x}^{p}} \biggr)
    \quad \text{as } |x| \to +\infty.
  \end{equation*}
\end{lemma}
The general statement of this Lemma which also applies in a nonradial
setting can be found in \cite{ST}. Here we report only the radial
version.

Now we can prove our compactness result:

\begin{lemma}\label{lemma-comp}
  The operator
  \begin{equation}
    M(w) := (-\Delta)^{-1}\left( \frac{w}{(1+|x|^2)^2} \right)
  \end{equation}
  is compact from $X$ to $X$.
\end{lemma}
\begin{proof}
  First of all, let us show that $M$ is well defined.
  If $w\in X$ then $|w|\le \norm{w}_D \, U$ so that
  $\frac{|w|}{(1+|x|^2)^2} \le C \norm{w}_D \, U^{\frac{N+2}{N-2}}
  \in L^\frac{2N}{N+2}(\R^N)$
  and, using the fact that
  $(-\Delta)^{-1} : L^{\frac{2N}{N+2}} \to D^{1,2}$, one gets $M(w)
  \in D^{1,2}$.  Moreover
  \begin{equation}\label{eq:M-bound}
    \abs{M(w)}
    \le C
    \int_{\R^N} \frac1{|x-y|^{N-2}} \frac{\abs{w(y)}}{(1+|y|^2)^2}
    \intd y
    \le C \norm{w}_D \int_{\R^N}\frac{U^{\frac{N+2}{N-2}}(y)}{|x-y|^{N-2}}
    = C \norm{w}_D \, U(x),
  \end{equation}
  and so $M(w) \in D$.  This argument incidentally shows that
  $M : X \to X$ is continuous.

  Now let
  $(w_n)$ be a bounded sequence in $X$ and let us prove that, up to a
  subsequence,
  $g_n := M(w_n)$ converges
  strongly to some $g\in X$.  On one hand, since $(w_n)$ is bounded in
  $D^{1,2}$, going if necessary to a subsequence, one can assume that
  $(w_n)$ converges weakly to some $w$ in $D^{1,2}$ and $w_n \to w$
  almost everywhere.  On the other hand, $( \norm{w_n}_D )$ is also
  bounded which means that $\abs{w_n} \le C U$ where $C$ is
  independent of $n$ and so
  $\frac{|w_n|}{(1 + \abs{x}^2)^2} \le C U^{\frac{N+2}{N-2}}$.
  Lebesgue's dominated convergence theorem then implies that
  $\frac{w_n}{(1 + \abs{x}^2)^2}$ converges strongly to
  $\frac{w}{(1 + \abs{x}^2)^2}$ in $L^{\frac{2N}{N+2}}$.  From the
  continuity of $(-\Delta)^{-1} : L^{\frac{2N}{N+2}} \to D^{1,2}$, one
  concludes that $g_n \to g$ in $D^{1,2}$.  Moreover, passing to the
  limit on $\abs{w_n} \le C U$ yields $w \in D$, and
  passing to the limit  on the inequality \eqref{eq:M-bound}
  for $w = w_n$ yields $g \in D$.

  It remains to show that $\norm{g_n - g}_D \to 0$.
  This is somewhat similar to the argument used in Lemma~\ref{lemma1}.
  First, H\"older inequality allows us to get the estimate:
  \begin{align}
    |g_n(x)-g(x)|
    &\le C\int_{\R^N} \frac{1}{|x-y|^{N-2}}\frac{|w_n(y)-w(y)|}{(1+|y|^2)^2}
    \nonumber\\
    &= C\int_{\R^N} \frac{U^{\frac{N+2}{N-2}-\e}(y)}{|x-y|^{N-2 }}
      \frac{|w_n(y)-w(y)|}{U^{1-\e}(y)}
    \nonumber\\
    \displaybreak[2]
    &\le C \left(\mkern 5mu \int_{\R^N}
      \left|\frac{U^{\frac{N+2}{N-2}-\e}(y)}{|x-y|^{N-2}}\right|^{\frac{q}{q-1}}
      \right)^{\frac{q-1}q}
      \cdot\left(\mkern 5mu
      \int_{\R^N}\left|\frac{|w_n(y)-w(y)|}{U^{1-\e}(y)}
      \right|^q\right)^{\frac 1q}
      \label{3.10}
  \end{align}
  where $\e > 0$ will be chosen small and $q$ large and satisfying
  $\e q=\frac{2N}{N-2}$.  Because of this latter constraint, the
  integrand of the right integral is bounded by
  $\norm{w_n - w}_D^q\, U^{\e q}(y) \le C U^{\e q}(y) \in L^1$ where $C$
  is independent of~$n$.  Lebesgue's dominated convergence theorem
  then implies that this integral converges to $0$ as $n \to \infty$.

  The proof will be complete if we show:
  \begin{equation}\label{cl1}
    \begin{split}
      \int_{\R^N}
      \left|\frac {U^{\frac{N+2}{N-2}-\e}(y)}{|x-y|^{N-2 }}
      \right|^{\frac{q}{q-1}} \intd y
      \le \frac{C}{(1+|x|)^{(N-2)\frac q{q-1}}}
      = C U^{\frac{q}{q-1}}.
    \end{split}
  \end{equation}
  This inequality follows from Lemma \ref{lemma-ST} because
  $h(y) = U^{\left(\frac{N+2}{N-2}-\e\right) \frac{q}{q-1}} \in L^1$
  if and only if $\bigl( \frac{N+2}{N-2}-\e \bigr) \frac{q}{q-1}
  > \frac{N}{N-2}$, which is possible if $\e$ is small enough and
  $q$ is large enough.
\end{proof}

\subsection{Application of the Crandall-Rabinowitz Theorem}

In this section we will verify the assumptions of the of the
Crandall-Rabinowitz Theorem.  Let us recall that by Corollary
\ref{corb}, the linearized system \eqref{linearization-2} has the
following radial solutions
\begin{itemize}
\item[i)] $(1,\dotsc,1)W$ (due to the
  dilation invariance of the problem), for every $\a$,
\item[ii)] $\eta:=e_n W_n(|x|)$
  where $e_n \ne 0$ satisfies $A(\bar \a)e_n = \lambda_n e_n$
  (see \eqref{eq:def-Wn})
  for $\bar \a$ satisfying \eqref{cond}.
\end{itemize}
Notice that $(1,\dots,1) \perp e_n$ and so $(1,\dots,1)W \perp \eta$
in $(D^{1,2})^k$, i.e., $\eta \in K$.

To apply Rabinowitz' result, we need to verify the assumptions of
Theorem~1.7 in~\cite{CRJFA}.  This is the purpose of the following
lemmas.

\begin{lemma}\label{ker-dzT}
  Let $T$ be as defined in \eqref{T} and assume that $\bar \a$
  satisfies \eqref{eq:ibar}--\eqref{y8}.
  Then $\ker\bigl(\partial_u T(\bar \a, U,\dots,U)\bigr)$ is
  one dimensional and it is given by
  \begin{equation}\label{et1}
    \ker\bigl(\partial_u T(\bar \a, U,\dots,U)\bigr)
    = \operatorname{span} \{\eta\}
    \qquad\text{where }
    \eta = e_n W_n ,
  \end{equation}
  $W_n$ is defined in \eqref{eq:def-Wn}, $e_n \ne 0$, and
  $A(\bar\a)e_n = \lambda_n e_n$.
\end{lemma}
\begin{proof}
  Let us consider the Fr\'echet derivative of $T$ at
  $(\a, U,\dots,U)$. We
  have that
  \begin{equation}
    \label{eq:dzT}
    \partial_u T(\a, U,\dots,U)
    \begin{pmatrix} w_1\\ \vdots\\w_k \end{pmatrix}
    =
    P_K
    \begin{pmatrix}
      \displaystyle
      w_i - (-\Delta)^{-1} \biggl(
      \sum_{j=1}^k a_{ij}(\a) \,
      \frac{N(N+2)}{(1+|x|^2)^2} \, w_j \biggr)
    \end{pmatrix}_{i=1}^k
  \end{equation}
  so that
  $\partial_u T(\bar \a,U,\dots,U) \left(\begin{smallmatrix} w_1\\
      \smallvdots \\ w_k\end{smallmatrix}\right)
  = \left(\begin{smallmatrix} 0\\
      \smallvdots\\
      0 \end{smallmatrix}\right)$ if and only if
  $(w_1,\dots,w_k)\in K$ is a solution~to
  \begin{equation}
    \label{3.12-b}
    \forall i=1,\dots,k,\qquad
    -\Delta w_i - 
    \frac{N(N+2)}{(1+|x|^2)^2} \sum_{j=1}^k a_{ij}(\bar\a) \, w_j
    = - L \Delta W
    \quad \text{in }\R^N,
  \end{equation}
  for some $L=L(w)\in \R$.  Multiplying by $W$, integrating, and
  summing up yields
  \begin{equation*}
    \sum_{i=1}^k \left( \mkern 5mu
      \int_{\R^N}\nabla w_i\cdot \nabla W\intd x
      - \int_{\R^N} \frac{N(N+2)} {(1+|x|^2)^2}
      \sum_{j=1}^k a_{ij}(\bar\a)  w_j(x) W(x)\intd x
    \right)
    = k L \int_{\R^N} \abs{\nabla W}^2 \intd x.
  \end{equation*}
  Recalling that $-\Delta W = \frac{N(N+2)}{(1+|x|^2)^2} W$ (see
  Corollary~\ref{corb}) and $\sum_i a_{ij} = 1$,
  one sees that the left hand side of the
  equation vanishes and so $L = 0$.  Thus, $w = (w_1,\dots, w_k)$
  is a solution to~\eqref{linearization-2} and, using again
  Corollary~\ref{corb} and assumptions \eqref{eq:ibar}--\eqref{y8},
  this is the case if and only if
  \begin{equation*}
    w \in \operatorname{span} \bigl\{ (1,\dots,1) W,\ \eta \bigr\}.
  \end{equation*}
  Recalling that $w \in K$, which means that $w$ is orthogonal to
  $(1,\dots,1) W$, and that $\eta \perp (1,\dots,1)W$, one concludes
  that $w$ is a multiple of~$\eta$.
\end{proof}

\begin{lemma}\label{ran-dzT}
  Under the assumptions of Lemma \ref{ker-dzT} the range
  $\Ran\bigl(\partial_u T(\bar \a, U,\dots,U)\bigr) \subseteq K$ has
  codimension one. It is the set of
  functions
  $f = (f_1,\dots,f_k)\in K$ that are orthogonal to $\eta$ in
  $\bigl(D^{1,2}(\R^N)\bigr)^k$, that is
  \begin{equation}\label{eq:ran-orthog}
    (f | \eta) :=
    \sum_{i=1}^k e_{n,i}
    \int_{\R^N} \nabla f_i \cdot \nabla W_n \intd x = 0
  \end{equation}
  where $e_n = (e_{n,i})_{i=1}^k$.
  Hence a complement of $\Ran\bigl(\partial_u T(\bar\a, U,\dots,U)\bigr)$ in
  $K$ is spanned by the vector $\eta$ defined in
  Lemma~\ref{ker-dzT}.
\end{lemma}
\begin{proof}
  This is a consequence of Lemma \ref{lemma-comp}.  Indeed the
  operator $\partial_uT(\bar\a, U,\dots,U)$ can be written
  \begin{math}
    \partial_u T(\bar\a, U,\dots,U)[w]
    = w - P_K \bigl( (-\Delta)^{-1} \sum_{j=1}^k a_{ij}(\bar\a)
    \frac{N(N+2)}{(1+|x|^2)^2} \, w_j \bigr)_{i=1}^k
  \end{math}
  because $w \in K$,
  and so is a compact perturbation of the
  identity. Thus \eqref{eq:ran-orthog} follows from the Fredholm Alternative.
\end{proof}

\begin{lemma}\label{transversality-dT}
  Under the assumptions of Lem\-ma~\ref{ker-dzT} and~\eqref{y9}, the
  operator $T$ satisfies
  \begin{equation}
    \label{T''}
    \partial_{\a u}T(\bar \a, U,\dots,U) [\eta]
    \notin \Ran\bigl(\partial_u T(\bar \a,U,\dots,U)\bigr)
  \end{equation}
  where $\eta$ is as defined in \eqref{et1}.
\end{lemma}
\begin{proof}
  The derivative $\partial_u T(\a, U,\dots,U)$ is given by~\eqref{eq:dzT}.
  Differentiating with respect to $\a$ yields
  \begin{equation*}
    \partial_{\a u}T(\bar \a, U,\dots,U) [\eta]
    = P_K g
  \end{equation*}
  where $g = (g_1,\dots, g_k)$ and
  \begin{equation*}
    g_i :=
    - (-\Delta)^{-1} \biggl(
    \sum_{j=1}^k  \partial_\a a_{ij}(\bar\a) \,
    \frac{N(N+2)}{(1+|x|^2)^2} \, \eta_j \biggr),
    \qquad i=1,\dots, k.
  \end{equation*}
  In view of Lemma~\ref{ran-dzT}, we have to show that
  $( P_Kg | \eta ) \ne 0$.  Since $\eta \in K$,
  $( P_Kg | \eta ) = (g|\eta)$.  Thus, we have to show that
  \begin{equation*}
    (g|\eta)
    = \sum_{i=1}^k e_{n,i}
    \int_{\R^N} \nabla\biggl(
    - (-\Delta)^{-1} \Bigl(
    \sum_{j=1}^k  \partial_\a a_{ij}(\bar\a) \,
    \frac{N(N+2)}{(1+|x|^2)^2} \, \eta_j
    \Bigr)
    \biggr) \cdot \nabla W_n \intd x \ne 0,
  \end{equation*}
  that is, recalling that $\eta_j = e_{n,j} \, W_n$,
  \begin{equation*}
    e_n \cdot \dd{A}{\a}(\bar\a) e_n
    \int_{\R^N}  \frac{N(N+2)}{(1+|x|^2)^2} W_n^2(x) \intd x \ne 0.
  \end{equation*}
  The proof is complete thanks to assumption~\eqref{y9}
  (see also~\eqref{eq:dA}).
\end{proof}

Now we are in position to apply the bifurcation result of \cite{CRJFA}:

\begin{proposition}\label{t3.2}
  Assume that $A$ satisfies \eqref{y1a} and \eqref{y2}. Assume further
  that there
  exists $\bar \a$ and $\ibar$ such that
  \eqref{eq:ibar}--\eqref{y9} are satisfied.
  Then the point $(\bar \a, U,\dots,U)$ is a radial bifurcation
  point for the curve $\alpha \mapsto (\a,U,\dots ,U)$
  of solutions to \eqref{y1}.
  More precisely, there exists continuous curves
  $\e \mapsto \a_{\e}$ and
  $\e \mapsto (u_{1,\e}, \dotsc, u_{k,\e})$,
  defined for $\e \in \R$ small enough, such that
  $\a_{0}=\bar\a$,
  $u_{i,0} = U$, and
  \begin{equation}\label{c1}
    \begin{cases}
      \displaystyle
      -\Delta u_{i,\e}
      = \sum_{j=1}^k a_{ij}(\a_{\e}) \, u_{j,\e}^{2^*-1}
      + L_\e
      \frac{N(N+2)}{(1 + \abs{x}^2)^2} \, W
      & \text{in }\R^N,\\[5\jot]
      u_{i,\e}>0,\quad u_{i,\e} \in D^{1,2}(\R^N),
    \end{cases}
  \end{equation}
  for some Lagrange multiplier $L_\e$. Moreover,
  for $\e$ small enough,
  \begin{equation}\label{fin}
    (u_{1,\e}, \dotsc, u_{k,\e})
    = (1,\dotsc, 1) \, U + \e e_n W_n\left(|x|\right)
    + \e \phi_{\e}(|x|)
  \end{equation}
  where $e_n$ is an eigenvector of $A$ for the eigenvalue
  $\L_\ibar(\bar\a) = \lambda_n$ and
  $\phi_{\e}$ is an uniformly bounded function in
  $\bigl(D^{1,2}(\R^N)\bigr)^k$ and such that $ \phi_{0}=0$.
\end{proposition}
\begin{proof}
  We apply Theorem 1.7 in \cite{CRJFA} at the operator $T$ defined in
  \eqref{T}. It is easy to see that $T(\a,U,\dots,U)=0$ for any $\a$.
  By Lemma~\ref{lemma1} the operators $\partial_{\a}T$,
  $\partial_uT$ and $\partial_{\a, u}T$
  are well defined and continuous from
  $\R \times (K\cap B^{k})$ to $K$.
  Lemma~\ref{ker-dzT} says that the kernel of
  $\partial_u T(\bar\a,U,\dots,U)$ is one-dimensional while
  Lemma~\ref{ran-dzT} implies that its range has
  codimension
  one.
  Finally, Lemma~\ref{transversality-dT} guarantees that
  the transversality condition holds.
  Therefore all assumptions of Theorem~1.7 in \cite{CRJFA} are
  satisfied.  As a consequence, there exists a neighborhood $V$ of
  $(\bar \a,U,\dots,U)$ in
  $\R \times (K\cap B^{k})$, an interval
  $(-\e_0,\e_0)$, and continuous functions
  $(-\e_0,\e_0)\to \R : \e \mapsto \a_{\e}$ and
  $(-\e_0,\e_0) \to B : \e \mapsto \phi_{i,\e}$ for
  $i=1,\dots,k$ such that $\a_{0} = \bar\a$, $\phi_{i,0}=0$ for
  $i=1,\dots,k$ and
  \begin{equation*}
    T^{-1}(\{0\})\cap V
    = \bigl\{ (\a,U,\dots,U) \bigm| (\a,U,\dots,U)\in V \bigr\}
    \cup
    \bigl\{ (\a_{\e}, u_{1,\e}, \dots, u_{k,\e})
    \bigm| |\e| < \e_0 \bigr\}
  \end{equation*}
  where $(u_{1,\e}, \dots, u_{k,\e})$ is defined by~\eqref{fin}.  In
  particular $T(\a_\e, u_{1,\e}, \dots, u_{k,\e}) = 0$ which means
  that $(u_{1,\e}, \dots, u_{k,\e})$ solves~\eqref{c1}. This concludes
  the proof.
\end{proof}

\begin{lemma}\label{c2}
 Let $L_\e$ be the Lagrange multiplier of Proposition \ref{t3.2}. Then
  \begin{equation}\label{d2}
    |L_\e|\le C.
  \end{equation}
\end{lemma}
\begin{proof}
  Let us use the function
  $W$ as test function
  in the first equation to \eqref{c1}. We get
  \begin{multline*}
    N(N+2) L_\e\int_{\R^N}
    \frac{W^2}{(1+|x|^2)^{2}} \intd x
    = \int_{\R^N}\nabla u_{1,\e}\cdot \nabla W  \intd x
    - \sum_{s=1}^ka_{1s}(\a) \int_{\R^N}u_{s,\e}^{2^*-1}
    W(x) \intd x.
  \end{multline*}
  The fact that $u_{s,\e}$ are uniformly bounded in
  $D^{1,2}(\R^N)$ then implies the claim.
\end{proof}

\subsection{The Pohozaev identity}
\label{Pohozaev}

\begin{proposition}\label{poho}
  Suppose that $(u_i)_{i = 1,\dotsc, k}$,
  are positive solutions in $D^{1,2}(\R^N)$ to
  \begin{equation}\label{P1}
    -\Delta u_i=\sum_{j=1}^k a_{ij} u_j^{2^*-1} + H_i(x)
    \quad\text{in } \R^N
  \end{equation}
  where $H_i$ are smooth functions satisfying
  \begin{equation}\label{P1-bis}
    H_i\in L^{2^*}(\R^N),\quad
    |x|H_i \in L^2(\R^N)
  \end{equation}
  and $A$ is an invertible symmetric matrix.  Let us write
  $A^{-1}=(a_{ij}^{-1})_{i,j=1,\dots,k}$.
  Then, the following Pohozaev identity holds
  \begin{equation}
    0=\sum_{i,j=1}^k a_{ij}^{-1} \int_{\R^N} H_i(x)
    \biggl(x\cdot \nabla u_j + \frac{N-2}{2} u_j \biggr)
    \intd x
  \end{equation}
\end{proposition}

\begin{proof}
  We will denote by $I_{i,R}$ various boundary terms on $\partial B_R$
  such that, for any integer~$i$,
  \begin{equation}\label{CR}
    |I_{i,R}|\le  C(N) R\int_{\partial B_R}
    \Biggl( \sum_{i,h=1}^k u_i u_h^{2^*-1}
    + \abs{\nabla u_i\cdot \nabla u_h} \Biggr).
  \end{equation}
  Set
  $\sum_{i,h=1}^ku_iu_h^{2^*-1} + \abs{\nabla u_i\cdot \nabla
  u_h} =: G(u_1,\dots,u_k)$
  and, as in \cite{BL}, let us show that there exists a sequence
  $R_n\to +\infty$ such that $I_{i,R_n}\rightarrow0$. Indeed since
  $u_i\in D^{1,2}(\R^N)$ we know that
  $G(u_1,\dots,u_k)\in L^{1}(\R^N)$ so that
  \begin{equation*}
    \int_0^ {+\infty}  \int_{\partial B_R} G(u_1,\dots,u_k )
    \intd\sigma \intd R
    < +\infty.
  \end{equation*}
  Hence, there exists a sequence $R_n\to +\infty$ such that
  \begin{equation}\label{asimp-CR}
    R_n\int_{\partial B_{R_n}} G(u_1,\dots,u_k)\intd\sigma\to 0
    \quad \text{as }n\to +\infty,
  \end{equation}
  and this shows that $I_{i,R_n}\to0$. From now, to simplify the
  notations, we agree that $R=R_n$ and we denote by $C(R)$ a linear
  combination of $I_{i,R}$.

  Let us now start our main argument with the identity:
  \begin{equation*}
    -\int_{B_R}\Delta u_i(x\cdot \nabla u_i)
    = \Bigl(1-\frac{N}{2}\Bigr) \int_{B_R}|\nabla u_i|^2\intd x
    - \underbrace{\int_{\partial B_R}  (x\cdot \nabla u_i)
      \frac{\partial u_i}{\partial \nu} }_{=I_{1,R}}
    + \frac 12 \underbrace{\int_{\partial B_R} |\nabla u_i|^2
      (x\cdot \nu)}_{=I_{2,R}}
  \end{equation*}
  Using the $i$-th equation in \eqref{P1}, we get
  \begin{equation*}\label{g1}
    \Bigl(1-\frac N2\Bigr) \int_{B_R}|\nabla u_i|^2\intd x+C(R)
    = \sum_{j=1}^ka_{ij} \int_{B_R} u_j^{\frac{N+2}{N-2}}
    (x\cdot \nabla u_i)\intd x
    + \int_{B_R}H_i(x) (x\cdot \nabla u_i)\intd x.
  \end{equation*}
  Next we estimate
  \begin{equation}\label{g3}
    \int_{B_R}u_j^\frac{N+2}{N-2}(x\cdot \nabla u_i)
    = -N\int_{B_R}u_iu_j^\frac{N+2}{N-2}-\frac{N+2}{N-2}
    \int_{B_R} u_iu_j^\frac{4}{N-2} (x\cdot \nabla u_j)
    + \underbrace{\int_{\partial B_R}
      u_iu_j^\frac{N+2}{N-2} (x\cdot \nu)}_{=I_{3,R}}
  \end{equation}
  which, when $i = j$, simplifies to
  \begin{equation}\label{g2}
    \int_{B_R}u_i^\frac{N+2}{N-2}(x\cdot \nabla u_i)=
    -\frac{N-2}2\int_{B_R}u_i^\frac{2N}{N-2}
    + \underbrace{\frac{N-2}{2N}R
      \int_{\partial B_R}u_i^\frac{2N}{N-2}}_{=I_{4,R}}
  \end{equation}
  Using~\eqref{g3} and \eqref{g2} we get
  \begin{multline}\label{cl0}
    \left(1-\frac N2\right)\int_{B_R}|\nabla u_i|^2\intd x+C(R)
    = -a_{ii} \frac{N-2}2\int_{B_R}u_i^\frac{2N}{N-2}
    -N\sum\limits_{j\ne i}a_{ij}\int_{B_R}u_iu_j^\frac{N+2}{N-2} \\
    -\frac{N+2}{N-2}\sum\limits_{j\ne i}a_{ij}
    \int_{B_R}(x\cdot \nabla u_j) u_iu_j^\frac{4}{N-2}
    + \int_{B_R}H_i(x) (x\cdot \nabla u_i) \intd x.
  \end{multline}
  On the other hand, multiplying equation \eqref{P1} by $u_i$ and
  integrating yields
  \begin{equation}\label{eq:mul-ui}
    \int_{B_R}|\nabla u_i|^2+C(R)
    = a_{ii}\int_{B_R}u_i^\frac{2N}{N-2}
    + \sum\limits_{j\ne i}a_{ij}\int_{B_R}u_iu_j^\frac{N+2}{N-2}
    + \int_{B_R}H_i(x)u_i \intd x.
  \end{equation}
  Summing \eqref{cl0} and \eqref{eq:mul-ui} multiplied by
  $\frac{N-2}{2}$ gives
  \begin{multline}
    \frac{N+2}2\sum\limits_{j\ne i}a_{ij}\int_{B_R}
    u_iu_j^{\frac{N+2}{N-2}} + \frac{N+2}{N-2}\sum_{j\ne i}
    a_{ij}\int_{B_R}(x\cdot \nabla u_j)
    u_iu_j^{\frac{4}{N-2}} + C(R)\\
    = \int_{B_R}H_i(x) \Bigl(x\cdot \nabla u_i + \frac{N-2}2u_i\Bigr)
  \end{multline}
  Setting
  \begin{equation}\nonumber
    A_{ij} := \frac{N+2}2\int_{B_R}u_iu_j^{\frac{N+2}{N-2}}
    + \frac{N+2}{N-2}\int_{B_R} (x\cdot \nabla u_j) u_iu_j^{\frac{4}{N-2}}
  \end{equation}
  and
  \begin{equation}\label{Bi}
    B_{ij} := \int_{B_R} H_i(x)
    \Bigl(x\cdot \nabla u_j+\frac{N-2}2u_j\Bigr),
  \end{equation}
  the previous identity becomes
  \begin{equation}\label{cl6}
    \boxed{
      \sum\limits_{j\ne i}a_{ij}A_{ij} = B_{ii} +C(R) .
    }
  \end{equation}
  Now let us use the factor $x\cdot\nabla u_h$ against $u_i$.
  Let us start with the identity:
  \begin{equation*}
    -\int_{B_R}\Delta u_i(x\cdot \nabla u_h)
    = \int_{B_R}\nabla u_i\cdot \nabla u_h
    + \sum_{\ell,m=1}^N \int_{B_R} x_\ell
    \frac{\partial u_i}{\partial x_m}
    \frac{\partial^2 u_h}{\partial x_\ell\partial x_m}
    - \underbrace{\int _{\partial B_R}
      \frac{\partial u_i}{\partial \nu}(x\cdot \nabla u_h)}_{=I_{5,R}}
  \end{equation*}
  Using \eqref{P1}, one gets
  \begin{align}
    & \int_{B_R}\nabla u_i\cdot \nabla u_h
      + \sum_{\ell,m=1}^N \, \int_{B_R} x_\ell \frac{\partial u_i}{\partial x_m}
      \frac{\partial^2 u_h}{\partial x_\ell \partial x_m} + C(R)
      \nonumber\\
    &\quad
      = -N\sum\limits_{j\ne h}a_{ij}\int_{B_R}u_h u_j^{\frac{N+2}{N-2}}
      -\frac{N+2}{N-2}\sum_{j\ne h}a_{ij}
      \int_{B_R}(x\cdot \nabla u_j)u_hu_j^{\frac 4{N-2}}
      - a_{ih}\frac{N-2}2\int_{B_R}u_h^{\frac{2N}{N-2}}
      \nonumber\\
    &\label{eq:ui-dh}\quad\hphantom{{}={}}
      + \int_{B_R} H_i(x)(x\cdot \nabla u_h) \intd x.
  \end{align}
  Our intention is to sum \eqref{eq:ui-dh} and the same expression
  with the indices $i$ and $h$ swapped.  Let us start by remarking
  that
  \begin{equation}\label{g10}
    \sum_{\ell,m=1}^N \int_{B_R} x_\ell
    \left(\frac{\partial u_i}{\partial x_m}
      \frac{\partial^2 u_h}{\partial x_\ell \partial x_m}
      + \frac{\partial u_h}{\partial x_m}
      \frac{\partial^2 u_i}{\partial x_\ell \partial x_m}\right)
    = -N\int_{B_R}\nabla u_i \cdot \nabla u_h
    + \underbrace{R\int_{\partial B_R}\nabla u_i\cdot\nabla u_h}_{=I_{6,R}}
    .
  \end{equation}
  Using \eqref{g10}, the sum of \eqref{eq:ui-dh} and its symmetric
  expression reads
  \begin{align}
    &(2-N)\int_{B_R}\nabla u_i\cdot\nabla u_h + C(R)  \nonumber\\
    &= -N \sum_{j\ne h}a_{ij}\int_{B_R} u_h u_j^{\frac{N+2}{N-2}}
      - \frac{N+2}{N-2}\sum_{j\ne h}a_{ij}
      \int_{B_R}(x\cdot \nabla u_j)u_hu_j^{\frac 4{N-2}}
      -a_{ih}\frac{N-2}2\int_{B_R}u_h^{\frac{2N}{N-2}}  \nonumber\\
    & \quad
      - N\sum_{j\ne i}a_{hj}\int_{B_R}u_i u_j^{\frac{N+2}{N-2}}
      - \frac{N+2}{N-2}\sum\limits_{j\ne i}a_{hj}
      \int_{B_R}(x\cdot \nabla u_j)u_iu_j^{\frac 4{N-2}}
      - a_{hi} \frac{N-2}2\int_{B_R}u_i^{\frac{2N}{N-2}}  \nonumber\\
    & \quad\label{cl4}
      + \int_{B_R} H_i(x)(x\cdot \nabla u_h)
      + \int_{B_R} H_h(x)(x\cdot \nabla u_i).
  \end{align}
  Using again the $i$-th equation of \eqref{P1} but this time
  multiplying by $u_h$ yields
  \begin{equation}\label{eq:di-dh}
    \int_{B_R}\nabla u_i\cdot \nabla u_h+C(R)
    = \sum_{j\ne h}a_{ij}\int_{B_R}u_hu_j^\frac{N+2}{N-2}
    + a_{ih}\int_{B_R}u_h^\frac{2N}{N-2}
    + \int_{B_R}H_i(x)u_h \intd x.
  \end{equation}
  Now, let us write
  $2\int_{B_R}\nabla u_i\cdot \nabla u_h = \int_{B_R}\nabla
  u_i\cdot\nabla u_h + \int_{B_R}\nabla u_h\cdot \nabla u_i$
  and substitute the first term using \eqref{eq:di-dh} and
  the second term using \eqref{eq:di-dh} with $i$ and $h$ swapped.
  Let us then multiply the resulting expression by $\frac{N-2}{2}$
  and add it to~\eqref{cl4}.  This gives the following equality:
  \begin{equation}\label{cl5}
    \begin{split}
      &\frac{N+2}{2} \sum_{j\ne h} a_{ij} \int_{B_R} u_h u_j^{\frac{N+2}{N-2}}
      + \frac{N+2}{N-2} \sum_{j\ne h}a_{ij}
      \int_{B_R} \bigl(x\cdot \nabla u_j\bigr) u_h u_j^{\frac 4{N-2}}\\
      &\qquad
      + \frac{N+2}{2} \sum_{j\ne i} a_{hj}\int_{B_R} u_i u_j^{\frac{N+2}{N-2}}
      + \frac{N+2}{N-2} \sum_{j\ne i}a_{hj}
      \int_{B_R} \bigl(x\cdot \nabla u_j\bigr) u_i u_j^{\frac 4{N-2}}\\
      &= \int_{B_R}H_i(x) \Bigl(x\cdot \nabla u_h + \frac{N-2}{2} u_h \Bigr)
      + \int_{B_R} H_h(x) \Bigl(x\cdot \nabla u_i + \frac{N-2}{2} u_i \Bigr)
      + C(R)
    \end{split}
  \end{equation}
  Recalling the definition of $A_{ij}$ and
  $B_{ij}$, one can write~\eqref{cl5} as
  \begin{equation}\label{cl7}
    \boxed{
      \sum_{j\ne h}a_{ij}A_{hj}+\sum_{j\ne i}a_{hj}A_{ij}
      = B_{ih} + B_{hi} + C(R) .}
  \end{equation}
  Now let us multiply \eqref{cl6} by $a_{ii}^{-1}$ for $i=1,\dots,k$
  and sum on $i$. We get that \eqref{cl6} becomes
  \begin{equation*}
    \sum_{i=1}^k\sum_{j\ne i}a_{ii}^{-1}a_{ij}A_{ij}
    = \sum_{i=1}^ka_{ii}^{-1}B_{ii} + C(R).
  \end{equation*}
  Multiplying \eqref{cl7} by
  $a_{ih}^{-1}$ and summing on the triangular bloc of the indices
  $(i,h)$ sa\-tis\-fy\-ing $1 \le i < h \le k$, we get
  \begin{equation*}
    \sum_{i=1}^{k-1}\sum_{h=i+1}^k
    a_{ih}^{-1} \biggl(\sum_{j\ne h}a_{ij}A_{hj}
    + \sum_{j\ne i}a_{hj}A_{ij} \biggr)
    = \sum_{i=1}^{k-1}\sum_{h=i+1}^k a_{ih}^{-1}(B_{ih}+B_{hi})
    + C(R).
  \end{equation*}
  Finally summing up the previous two relations, we get
  \begin{multline}\label{cl8}
    \sum_{i=1}^{k-1}\sum_{h=i+1}^ka_{ih}^{-1}
    \biggl(\sum_{j\ne h}a_{ij}A_{hj}
    + \sum_{j\ne i}a_{hj}A_{ij} \biggr)
    + \sum_{i=1}^k\sum_{j\ne i}a_{ii}^{-1}a_{ij}A_{ij}\\
    = \sum_{i=1}^{k-1}\sum_{h=i+1}^ka_{ih}^{-1}\left(B_{ih}+B_{hi}\right)
    + \sum_{i=1}^k a_{ii}^{-1} B_{ii} + C(R).
  \end{multline}
  Let us consider the RHS of \eqref{cl8} and observe that, using the
  symmetry of the matrix $A^{-1}$, we have
  \begin{align*}
    &\hspace{-3em}
      \sum_{i=1}^{k-1}\sum_{h=i+1}^ka_{ih}^{-1} (B_{ih}+B_{hi})
      + \sum_{i=1}^k a_{ii}^{-1} B_{ii}\\
    & = \sum_{i=1}^{k-1}\sum_{h=i+1}^ka_{ih}^{-1} B_{ih}
      + \sum_{h=1}^{k-1} \sum_{i=h+1}^{k} a_{ih}^{-1}B_{ih}
      + \sum_{i=1}^ka_{ii}^{-1}B_{ii}
      = \sum_{i=1}^{k}\sum_{h=1}^{k} a_{ih}^{-1} B_{ih}
  \end{align*}
  where the last equality results from the fact that all $(i,h) \in
  \{1,\dotsc, k\}^2$ are present in the previous terms: all $i < h$ in
  the first double sum, $i > h$ in the second one, and $i = h$ in the
  third sum.

  Doing the same kind of computation for the LHS, we have
  \begin{multline*}
    \sum_{i=1}^{k-1}\sum_{h=i+1}^ka_{ih}^{-1}
    \biggl(\sum_{j\ne h}a_{ij}A_{hj}
    + \sum_{j\ne i}a_{hj}^{-1}A_{ij} \biggr)
    + \sum_{i=1}^k\sum_{j\ne i} a_{ii}^{-1} a_{ij} A_{ij}\\
    = \sum_{i=1}^{k}\sum_{j\ne i} A_{ij}
    \biggl(\sum_{h=1}^{k}a_{ih}^{-1}a_{hj} \biggr)
    = \sum_{i=1}^{k}\sum_{j\ne i}A_{ij}\delta^j_i
    = 0
  \end{multline*}
  Therefore, \eqref{cl8} reads
  \begin{equation*}
    \sum_{i=1}^{k}\sum_{h=1}^{k} a_{ih}^{-1} B_{ih} + C(R)
    = 0.
  \end{equation*}
  From the summability assumptions on $u_i$ and $H$, we can pass to the
  limit along the sequence $R_n\to +\infty$ chosen at the beginning of
  this proof and get
  \begin{equation*}
    \sum_{i=1}^{k}\sum_{h=1}^{k} a_{ih}^{-1} \int_{\R^N}
    H_i(x) \Bigl(x\cdot \nabla u_h + \frac{N-2}{2} u_h \Bigr)
    = 0.
    \qedhere
  \end{equation*}
\end{proof}

\begin{lemma}\label{sum-invA}
  Le $A$ be an invertible matrix satisfying~\eqref{y1a}, \eqref{y2}
  and denote $a_{ij}^{-1}$ the entries of $A^{-1}$.  Then
  \begin{equation*}
    \sum_{i,j=1}^k a_{ij}^{-1} = k.
  \end{equation*}
\end{lemma}
\begin{proof}
  Assumption~\eqref{y2} can be written
  \begin{math}
    A (1,\dots,1)
    = (1,\dots,1)
  \end{math}.
  Multiplying both sides by $(1,\dots,1)^\top A^{-1}$, one gets
  \begin{equation*}
    \sum_{i,j=1}^k a_{ij}^{-1}
    = (1,\dots,1)^\top A^{-1} (1,\dots,1)
    = (1,\dots,1)^\top (1,\dots,1)
    = k.
    \qedhere
  \end{equation*}
\end{proof}

We are in position to prove our main result:
\begin{proof}[Proof of Theorem \ref{Th1}]
  Proposition~\ref{t3.2} says that there exist $u_{j,\e}$
  satisfying~\eqref{c1} for $\e$ small enough.
  Assumption~\eqref{y10*} says that the
  matrix $A$ is invertible at $\bar \a$ and then from
  Proposition~\ref{poho}, we get
  \begin{equation}\label{lll}
    L_{\e} \sum_{i,j=1}^k a_{ij}^{-1} \int_{\R^N}
    \frac{N(N+2)}{(1+|x|^2)^{2}} W(x)
    \Bigl(x\cdot \nabla u_{i,\e} + \frac{N-2}{2} u_{i,\e} \Bigr)
    \intd x = 0.
  \end{equation}
  Recalling that $u_{i,\e}\to U$ in $D^{1,2}(\R^N)$ when $\e \to 0$,
  one can pass to the limit and get
  \begin{equation*}
    \int_{\R^N}\frac {N(N+2)}{(1+|x|^2)^{2}} W(x)
    \Bigl(x\cdot \nabla u_{i,\e}+\frac{N-2}2u_{i,\e}\Bigr) \intd x
    \xrightarrow[\e \to 0]{}
    \int_{\R^N} \frac{N(N+2)}{(1+|x|^2)^{2}} W^2(x) \intd x
    \ne 0.
  \end{equation*}
  Thanks to Lemma~\ref{sum-invA}, $\sum_{i,j=1}^k a_{ij}^{-1} \ne 0$
  and then \eqref{lll} implies that $L_\e=0$ for
  $\e$ small enough, concluding the proof.
\end{proof}

\appendix

\section{Computation of the first derivative
  of  the parameter}\label{A}
In this appendix we give some information on the behavior of branch of solutions of Theorem \ref{Th1}. Let us recall that the bifurcation is called {\em transcritical} if 
\begin{equation}\label{A1}
  \dd{\a_{\e}}{\e}\Bigr|_{\e = 0} \ne 0.
\end{equation}
Although in the literature are present formulas for the calculation of the derivative of  $\alpha_\e$ (see for example \cite{K}), it seems difficult to provide a complete characterization of the bifurcation diagram. In next proposition we give a sufficient condition to have a transcritical bifurcation.
\begin{proposition}
Let us suppose that
\begin{equation}\label{A2}
 \sum_{j=1}^k e_{n,j}^3 \int_{\R^N} U^{\frac{6-N}{N-2}} \, W_n^3 \intd x\ne0,
\end{equation}
 (see~\eqref{et1} for the definition of $e_n$). Then the bifurcation given in Theorem \ref{Th1} is transcritical.
\end{proposition}
\begin{remark}
If $k=2$ we have $e_{n}=(1,-1)$ and so \eqref{A2} is never satisfied. In this case we need refined estimates involving higher order derivatives. We do not investigate this situation. On the other hand, if $k\ge3$ it is easy to find matrices $A$ verifying $ \sum_{j=1}^k e_{n,j}^3\ne0$. Finally, in the special case $N=4$ and $n=2$ we get
\begin{equation}
 \int_{\R^N} U^{\frac{6-N}{N-2}} \, W_n^3 \intd x=\int_{\R^4} U \, W_2^3 \intd x=  \int_{-1}^1 (1 - \xi^2)
    \bigl( P_2^{(1, 1)}(\xi) \bigr)^3 \intd \xi= \frac{27}{64} \int_{-1}^1 (1 - \xi^2) (5\xi^2-1)^3\ne0.
\end{equation}
\end{remark}
\begin{proof}

Using the formula (1.6.3) at page 21 in \cite{K} we get,
\begin{equation}\label{x1}
  \dd{\a_{\e}}{\e}\Bigr|_{\e = 0}
  = - \frac12\frac{\bigl(\partial^2_{u} T(\bar\a,U,\dots,U) [\eta,\eta]
    \bigm| \eta \bigr)}
  {\bigl(\partial_{\a u}T(\bar\a,U,\dots,U) [\eta] \bigm| \eta \bigr)}
\end{equation}
where $\eta$ is defined in~\eqref{et1}
and $(\cdot\mathbin|\cdot)$ denotes the inner product
in $\bigl(D^{1,2}(\R^N)\bigr)^k$. Let
us compute the numerator.
Given the definition~\eqref{T} of $T$, one easily gets
\begin{equation*}
  \partial^2_{u} T(\a, U,\dots,U) [w, w'\thinspace ]
  = P_K
  \begin{pmatrix}
    \displaystyle
    - (-\Delta)^{-1} \biggl(
    \sum_{j=1}^k a_{ij}(\a) \, U_2 \, w_j w'_j \biggr)
  \end{pmatrix}_{i=1}^k
\end{equation*}
where $U_2 := \frac{4(N+2)}{(N-2)^2} U^{\frac{6-N}{N-2}}$.
In particular, in view of the definition of $\eta$ (see~\eqref{et1}),
\begin{equation*}
  \partial^2_{u} T (\bar\a, U,\dots,U) [\eta, \eta]
  = P_K
  \begin{pmatrix}
    \displaystyle
    - (-\Delta)^{-1} \biggl(
    \sum_{j=1}^k a_{ij}(\bar\a) \, e_{n,j}^2 \, U_2 \, W_n^2 \biggr)
  \end{pmatrix}_{i=1}^k
\end{equation*}
Since $\eta \in K$, we can drop $P_K$ when performing the inner
product.  Thus the numerator reads:
\begin{multline}\label{eq:DzzT=0}
  \sum_{i=1}^k e_{n,i}
  \int_{\R^N} \nabla\Biggl(
  - (-\Delta)^{-1} \biggl(
  \sum_{j=1}^k a_{ij}(\bar\a) \, e_{n,j}^2 \, U_2 \, W_n^2 \biggr)  
  \Biggr)
  \cdot \nabla W_n \intd x\\
  =-
  e_{n,i} \, a_{ij}(\bar\a) \, e_{n,j}^2
  \int_{\R^N} U_2 \, W_n^3 \intd x
\end{multline}
Recalling that $e_n$ is an eigenvector of $A(\bar\a)$ for the
eigenvalue $\L_\ibar(\bar\a)$, one gets from~\eqref{eq:DzzT=0} 
\begin{equation*}
   \dd{\a_{\e}}{\e}\Bigr|_{\e = 0}=0\Leftrightarrow - \L_\ibar(\bar\a) \sum_{j=1}^k e_{n,j}^3
  \int_{\R^N} U_2 \, W_n^3 \intd x=0
\end{equation*}
which gives the claim.
\end{proof}

\end{document}